\documentclass[12pt]{article}

\usepackage{amssymb,amsmath,amsfonts,amssymb}
\usepackage{graphics,graphicx,color}

\def\@abssec#1{\vspace{.05in}\footnotesize \parindent .2in
{\bf #1. }\ignorespaces}

\graphicspath{{/EPSF/}{Figures/}}
\DeclareGraphicsExtensions{.eps}

\newtheorem{theorem}{Theorem}[section]

\newtheorem{proposition}[theorem]{Proposition}

\newtheorem{remark}[theorem]{Remark}

\def \Rm {\mathbb R}

\def \Sm {\mathbb S}

\newcommand{\eps}{\varepsilon}

\newcommand{\dsum}{\displaystyle\sum}
\newcommand{\dint}{\displaystyle\int}
\newcommand{\pdr}[2]{\dfrac{\partial{#1}}{\partial{#2}}}
\newcommand{\pdrr}[2]{\dfrac{\partial^2{#1}}{\partial{#2}^2}}

\newcommand{\aver}[1]{\langle {#1} \rangle}

\newcommand{\be}{\mathbf e}

\newcommand{\bp}{\mathbf p} \newcommand{\bq}{\mathbf q}
 \newcommand{\bv}{\mathbf v}
\newcommand{\bw}{\mathbf w}

\newcommand{\cout}[1]{}

\newcommand{\cO}{{\mathcal O}}

\newcommand{\mg}{{\mathfrak g}}
\newcommand{\mh}{{\mathfrak h}}
\newcommand{\mk}{{\mathfrak k}}

 \renewcommand{\arraystretch}{1.5}

\title{Cauchy problem for Ultrasound Modulated EIT}

\author{Guillaume Bal \thanks{Department of Applied Physics and 
        Applied Mathematics, Columbia University, 
        New York NY, 10027; gb2030@columbia.edu}}

\begin{document}
 
\maketitle


\begin{abstract}
  Ultrasound modulation of electrical or optical properties of materials offers the possibility to devise hybrid imaging techniques that combine the high electrical or optical contrast observed in many settings of interest with the high resolution of ultrasound. Mathematically, these modalities require that we reconstruct a diffusion coefficient $\sigma(x)$ for $x\in X$, a bounded domain in $\Rm^n$, from knowledge of $\sigma(x)|\nabla u|^2(x)$ for $x\in X$, where $u$ is the solution to the elliptic equation $-\nabla\cdot\sigma\nabla u=0$ in $X$ with $u=f$ on $\partial X$. 
  
This inverse problem may be recast as a nonlinear equation, which formally takes the form of a $0$-Laplacian. Whereas $p-$Laplacians with $p>1$ are well-studied variational elliptic non-linear equations, $p=1$ is a limiting case with a convex but not strictly convex functional, and the case $p<1$ admits a variational formulation with a functional that is not convex. In this paper, we augment the equation for the $0$-Laplacian with full Cauchy data at the domain's boundary, which results in a, formally overdetermined, nonlinear  hyperbolic equation. 

The paper presents existence, uniqueness, and stability results for the Cauchy problem of the $0$-Laplacian. In general, the diffusion coefficient $\sigma(x)$ can be stably reconstructed only on a subset of $X$ described as the domain of influence of the space-like part of the boundary $\partial X$ for an appropriate Lorentzian metric. Global reconstructions for specific geometries or based on the construction of appropriate complex geometric optics solutions are also analyzed.
\end{abstract}
 

\renewcommand{\thefootnote}{\fnsymbol{footnote}}
\renewcommand{\thefootnote}{\arabic{footnote}}

\renewcommand{\arraystretch}{1.1}





\section{Introduction}

Electrical Impedance Tomography (EIT) and Optical Tomography (OT) are medical imaging modalities that take advantage of the high electrical and optical contrast exhibited by different tissues, and in particular the high contrast often observed between healthy and non-healthy tissues. Electrical potentials and photon densities are modeled in such applications by a diffusion equation, which is known not to propagate singularities, and as a consequence the reconstruction of the diffusion coefficient in such modalities often comes with a poor resolution \cite{AS-IP-10,B-IP-09,U-IP-09}.

Ultrasound modulations have been proposed as a means to combine the high contrast of EIT and OT with the high resolution of ultrasonic waves propagating in an essentially homogeneous medium  \cite{W-JDM-04}. In the setting of EIT, ultrasound modulated electrical impedance tomography (UMEIT), also called acousto-electric tomography, has been proposed and analyzed in \cite{ABCTF-SIAP-08,BBMT-11,CFGK-SJIS-09,GS-SIAP-09,KK-AET-11,ZW-SPIE-04}. In the setting of optical tomography, a similar model of ultrasound modulated tomography (UMOT), also called acousto-optic tomography, has been derived in \cite{BS-PRL-10} in the so-called incoherent regime of wave propagation, while a large physical literature deals with the coherent regime \cite{AFRBG-OL-05,KLZZ-JOSA-97,W-JDM-04}, whose mathematical structure is quite different.

\paragraph{Elliptic forward problem.}
In the setting considered in this paper, both UMEIT and UMOT are inverse problems aiming to reconstruct an unknown coefficient $\sigma(x)$ from knowledge of a functional of the form $H(x)=\sigma(x)|\nabla u|^2(x)$, where $u(x)$ is the solution to the elliptic equation:
\begin{equation}
\label{eq:elliptic}
  -\nabla\cdot \sigma(x)\nabla u = 0\quad \mbox{ in } X, \qquad u=f\quad \mbox{ on } \partial X.
\end{equation}
Here,  $X$ is an open bounded domain in $\Rm^n$ with spatial dimension $n\geq2$. We denote by $\partial X$ the boundary of $X$ and by $f(x)$ the Dirichlet boundary conditions prescribed in the physical experiments. Neumann or more general Robin boundary conditions could be analyzed similarly. We assume that the unknown diffusion coefficient $\sigma$  is a real-valued, scalar, function defined on $X$. It is bounded above and below by positive constants and assumed to be (sufficiently) smooth. The coefficient $\sigma(x)$ models the electrical conductivity in the setting of electrical impedance tomography (EIT) and the diffusion coefficient of particles (photons) in the setting of optical tomography (OT). Both EIT and OT are high contrast modalities. We focus on the EIT setting here for concreteness and refer to $\sigma$ as the conductivity.

The derivation of such functionals $H(x)$ from physical experiments, following similar derivations in \cite{BBMT-11,BS-PRL-10,KK-AET-11}, is recalled in section \ref{sec:deriv}. For a derivation based on the focusing of acoustic pulses (in the time domain), we refer the reader to \cite{ABCTF-SIAP-08}. This problem has been considered numerically in \cite{ABCTF-SIAP-08,GS-SIAP-09,KK-AET-11}. In those papers, it is shown numerically that UMEIT allows for high-resolution reconstructions although typically more information than one measurement of the form $H(x)=\sigma(x)|\nabla u|^2(x)$ is required.

Following the methodology in \cite{CFGK-SJIS-09} where the two dimensional setting is analyzed, \cite{BBMT-11} analyzes the reconstruction of $\sigma$ in UMEIT from multiple measurements at least equal to the spatial dimension $n$. The stability estimates obtained in \cite{BBMT-11} show that the reconstructions in UMEIT are indeed very stable with respect to perturbations of the available measurements. Such results are confirmed by the theoretical investigations in a linearized setting and the numerical simulations proposed in \cite{KK-AET-11}. In this paper, we consider the setting where a unique measurement $H(x)=\sigma(x)|\nabla u|^2(x)$ is available.

\paragraph{The inverse problem as a $p-$Laplacian.}
Following \cite{ABCTF-SIAP-08,BS-PRL-10,GS-SIAP-09}, we recast the inverse problem in UMEIT as a nonlinear partial differential equation; see \eqref{eq:Cauchy1} below. This equation is formally an extension to the case $p=0$ of the $p-$Laplacian elliptic equations
\begin{displaymath} 
   -\nabla\cdot \dfrac{H(x)}{|\nabla u|^{2-p}} \nabla u=0, 
\end{displaymath}
posed on a bounded, smooth, open domain $X\subset\Rm^n$, $n\geq2$, with prescribed Dirichlet conditions, say. When $1<p<\infty$, the above problem is known to admit a variational formulation with convex functional $J[\nabla u]=\int_X H(x) |\nabla u|^{p}(x) dx$, which admits a unique minimizer, in an appropriate functional setting, solution of the above associated Euler-Lagrange equation \cite{evans}.

The case $p=1$ is a critical case as the above functional remains convex but not strictly convex. Solutions are no longer unique in general. This problem has been extensively analyzed in the context of EIT perturbed by magnetic fields (CDII and MREIT) \cite{KWYS-IEEE-02,NTT-IP-07,NTT-IP-09}, where it is shown that slight modifications of the $1-$Laplacian admit unique solutions in the setting of interest in MREIT. Of interest for this paper is the remark that the reconstruction when $p=1$ exhibits some locality, in the sense that local perturbations of the source and boundary conditions of the $1-$Laplacian do not influence the solution on the whole domain $X$. This behavior is characteristic of a transition from an elliptic equation when $p>1$ to a hyperbolic equation when $p<1$.

\paragraph{The inverse problem as a hyperbolic nonlinear equation.}

When $p<1$, the above functional $J[\nabla u]$ is no longer convex. When $p=0$, it should formally be replaced by $J[\nabla u]=\int_X H(x)\ln|\nabla u| (x)dx$, whose Euler-Lagrange equation is indeed \eqref{eq:Cauchy1} below. The resulting $0-$Laplacian is not an elliptic problem. As we mentioned above, it should be interpreted as a hyperbolic equation as the derivation of \eqref{eq:Cauchy2} below indicates. 

Information then propagates in a local fashion. Moreover, compatible boundary conditions need to be imposed in order for the hyperbolic equation to be well-posed \cite{H-SP-97,Taylor-PDE-1}. We thus augment the nonlinear equation with Cauchy boundary measurements, i.e., Dirichlet and Neumann boundary conditions simultaneously. As we shall see in the derivation of UMEIT in the next section, imposing such boundary conditions essentially amounts to assuming that $\sigma(x)$ is known at the domain's boundary. This results in an overdetermined problem in the same sense as a wave equation with Cauchy data at time $t=0$ and at time $t=T>0$ is overdetermined. Existence results are therefore only available in a local sense. We are primarily interested in showing a uniqueness (injectivity) result, which states that at most one coefficient $\sigma$ is compatible with a given set of measurements, and a stability result, which characterizes how errors in measurements translate into errors in reconstructions. Redundant measurements then clearly help in such analyses.

\paragraph{Space-like versus time-like boundary subsets.} Once UMEIT is recast as a hyperbolic problem,  we face several difficulties. The equation is hyperbolic in the sense that one of the spatial variables plays the usual role of ``time'' in a second-order wave equation. Such a ``time'' variable has an orientation that depends on position $x$ in $X$ and also on the solution of the hyperbolic equation itself since the equation is nonlinear. Existence and uniqueness results for such equations need to be established, and we shall do so in sections \ref{sec:local} and \ref{sec:global} below adapting known results on linear and nonlinear hyperbolic equations that are summarized in \cite{H-SP-97,Taylor-PDE-1}.

More damaging for the purpose of UMEIT and UMOT is the fact that hyperbolic equations propagate information in a stable fashion only when such information enters through a space-like surface, i.e., a surface that is more orthogonal than it is tangent to the direction of ``time''.  In two dimensions of space, the time-like and space-like variables can be interchanged so that when $n=2$, unwanted singularities can propagate inside the domain only through points with ``null-like'' normal vector, and in most settings, such points have (surface Lebesgue) zero-measure. In $n=2$, it is therefore expected that spurious instabilities may propagate along a finite number of geodesics and that the reconstructions will be stable otherwise. 

In dimensions $n\geq3$, however, a large part of the boundary $\partial X$ will in general be purely ``time-like'' so that the information available on such a part of the surface cannot be used to solve the inverse problem in a stable manner \cite{H-SP-97}. Only on the domain of influence of the space-like part of the boundary do we expect to stably solve the nonlinear hyperbolic equation, and hence reconstruct the unknown conductivity $\sigma(x)$. 

\paragraph{Special geometries and special boundary conditions.} As we mentioned earlier, the partial reconstruction results described above can be improved in the setting of multiple measurements. Once several measurements, and hence several potential ``time-like'' directions are available, it becomes more likely that $\sigma$ can be reconstructed on the whole domain $X$. In the setting of well-chosen multiple measurements, the theories developed in \cite{BBMT-11,CFGK-SJIS-09} indeed show that $\sigma$ can be uniquely and stably reconstructed on $X$. 

An alternative solution is to devise geometries of $X$ and of the boundary conditions that guarantee that the ``time-like'' part of the boundary $\partial X$ is empty. Information can then be propagated uniquely and stably throughout the domain. In section \ref{sec:global}, we consider several such geometries. The first geometry consists of using an annulus-shaped domain and to ensure that the two connected components of the boundary are level-sets of the solution $u$. In such situations, the whole boundary $\partial X$ turns out to be ``space-like''. Moreover, so long as $u$ does not have any critical point, we can show that the reconstruction can be stably performed on the whole domain $X$. 

Unfortunately, only in dimension $n=2$ can we be sure that $u$ does not have any critical point independent of the unknown conductivity $\sigma$. This is based on the fact that critical points in a two-dimensional elliptic equation are necessary isolated as used e.g. in \cite{A-AMPA-86} and our geometry simply prevents their existence. In three dimensions of space, however, critical points can arise. Such results are similar to those obtained in \cite{BMN-ARMA-04} in the context of homogenization theory and are consistent with the analysis of critical points in elliptic equations as in, e.g., \cite{CF-JDE-85,HaHoHoNa-JDG99}.

In dimension $n\geq3$, we thus need to use another strategy to ensure that one vector field is always available for us to penetrate information inside the domain in a unique and stable manner. In this paper, such a result is obtained by means of boundary conditions $f$ in \eqref{eq:elliptic} that are ``close'' to traces of appropriate complex geometric optics (CGO) solutions, which can be constructed provided that $\sigma(x)$ is sufficiently smooth. The CGO solutions are used to obtain required qualitative properties of the solutions to linear elliptic equations as it was done in the setting of other hybrid medical imaging modalities in, e.g., \cite{BR-IP-11,BRUZ-IP-11,BU-IP-10,T-IP-10}; see also the review paper \cite{B-IO-12}.

\medskip

The rest of the paper is structured as follows. Section \ref{sec:deriv} presents the derivation of the functional $H(x)=\sigma(x)|\nabla u|^2$ from ultrasound modulation of a domain of interest and the transformation of the inverse problem as a nonlinear hyperbolic equation. In section \ref{sec:local}, local results of uniqueness and stability are presented adapting results on linear hyperbolic equations summarized in \cite{Taylor-PDE-1}. These results show that UMEIT and UMOT are indeed much more stable modalities than EIT and OT. The section concludes with a local reconstruction algorithm, which shows that the nonlinear equation admits a solution even if the available data are slightly perturbed by, e.g., noise. The existence result is obtained after an appropriate change of variables from the result for time-dependent second-order nonlinear hyperbolic equations in \cite{H-SP-97}. Finally, in section \ref{sec:global}, we present global uniqueness and stability results for UMEIT for specific geometries or specific boundary conditions constructed by means of CGO solutions.

%
\section{Derivation of a non-linear equation}
\label{sec:deriv}

\paragraph{Ultrasound modulation.}
A methodology to combine high contrast with high resolution consists of perturbing the diffusion coefficient acoustically. Let an acoustic signal propagate throughout the domain. We assume here that the sound speed is constant and that the acoustic signal is a plane wave of the form $p\cos(k\cdot x + \varphi)$ where $p$ is the amplitude of the acoustic signal, $k$ its wave-number, and $\varphi$ an additional phase. The acoustic signal modifies the properties of the diffusion equation. We assume that such an effect is small but measurable and that the  coefficient in \eqref{eq:elliptic} is modified as 
\begin{equation}
  \label{eq:modifcoefs}
  \sigma_\eps(x) = \sigma(x) (1+\eps \cos(k\cdot x + \varphi) ),
\end{equation}
where $\eps=p\Gamma$ is the product of the acoustic amplitude $p\in\Rm$ and a measure $\Gamma>0$ of the coupling between the acoustic signal and the modulations of the constitutive parameter in \eqref{eq:elliptic}. For more information about similar derivations, we refer the reader to \cite{ABCTF-SIAP-08,BS-PRL-10,KK-AET-11}.

Let $u$ be a solution of \eqref{eq:elliptic} with fixed boundary condition $f$. When the acoustic field is turned on, the coefficients are modified as described in \eqref{eq:modifcoefs} and we denote by $u_\eps$ the corresponding solution. Note that $u_{-\eps}$ is the solution obtained by changing the sign of $p$ or equivalently by replacing $\varphi$ by $\varphi+\pi$.

By the standard continuity of the solution to \eqref{eq:elliptic} with respect to changes in the coefficients and regular perturbation arguments, we find that $u_\eps=u_0+\eps u_{1} + O(\eps^2)$. Let us multiply the equation for $u_\eps$ by $u_{-\eps}$ and the equation for $u_{-\eps}$ by $u_\eps$, subtract the resulting equalities, and use standard integrations by parts. We obtain that 
\begin{equation}
  \label{eq:Green}
  \dint_X (\sigma_\eps-\sigma_{-\eps})\nabla u_\eps\cdot\nabla u_{-\eps} dx = \dint_{\partial X}  \sigma_{-\eps} \pdr{u_{-\eps}}{\nu} u_\eps -\sigma_\eps \pdr{u_\eps}{\nu} u_{-\eps}  d\sigma.
\end{equation}
Here, $\nu(x)$ is the outward unit normal to $X\subset\Rm^n$ at $x\in\partial X$ and as usual $\frac{\partial}{\partial\nu}\equiv\nu\cdot\nabla$. We assume that $\sigma_\eps \partial_\nu u_\eps$ is measured on $\partial X$, at least on the support of $u_{\eps}=f$ for all values $\eps$ of interest. Note that the above equation still holds if the Dirichlet boundary conditions are replaced by Neumann (or more general Robin) boundary conditions. Let us define
\begin{equation}
  \label{eq:measbdry}
  J_\eps :=  \dfrac12\dint_{\partial X}  \sigma_{-\eps} \pdr{u_{-\eps}}{\nu} u_\eps -\sigma_\eps \pdr{u_\eps}{\nu} u_{-\eps}  d\sigma  \,\,=\,\, \eps J_1 + O(\eps^3).
\end{equation}
The term of order $O(\eps^2)$ vanishes by symmetry. We assume that the real valued functions $J_1=J_1(k,\varphi)$ are known. Such knowledge is based on the physical boundary measurement of the Cauchy data $(u_\eps,\sigma_\eps \partial_\nu u_\eps)$  on $\partial X$. 

Equating like powers of $\eps$, we find at the leading order that:
\begin{equation}
  \label{eq:leadingorder}
  \dint_X \big[\sigma(x) \nabla u_0\cdot\nabla u_0(x) \big] \cos(k\cdot x+\varphi) dx = J_1(k,\varphi).
\end{equation}
This may be acquired for all $k\in\Rm^n$ and $\varphi=0,\frac\pi2$, and hence provides the Fourier transform of
\begin{equation}
  \label{eq:H0Lap}
  H(x) = \sigma(x) |\nabla u_0|^2(x).
\end{equation}
Upon taking the inverse Fourier transform of the measurements \eqref{eq:leadingorder}, we thus obtain the internal functional \eqref{eq:H0Lap}.

\paragraph{Nonlinear hyperbolic inverse problem.}
The forward problem consists of assuming $\sigma$ and $f(x)$ known, solving \eqref{eq:elliptic} to get $u(x)$, and then constructing $H(x)=\sigma(x)|\nabla u|^2(x)$. The inverse problem consists of reconstructing $\sigma$ and $u$ from knowledge of $H(x)$ and $f(x)$. 

As we shall see, the linearization of the latter inverse problem may involve an operator that is not injective and so there is no guaranty that $u$ and $\sigma$ can be uniquely reconstructed; see Remark \ref{rem:linear} below. In this paper, we instead assume that the Neumann data $\sigma\nu\cdot\nabla u$ and the conductivity $\sigma(x)$ on $\partial X$ are also known. We saw that measurements of Neumann data were necessary in the construction of $H(x)$, and so our main new assumption is that $\sigma(x)$ is known on $\partial X$. This allows us to have access to $\nu\cdot\nabla u$ on $\partial X$. 
 
Combining \eqref{eq:elliptic} and \eqref{eq:H0Lap} with the above hypotheses, we can eliminate $\sigma$ from the equations and obtain the following Cauchy problem for $u(x)$:
\begin{equation}
  \label{eq:Cauchy1}
   -\nabla \cdot \dfrac{H(x)}{|\nabla u|^2(x)} \nabla u =0\,\mbox{ in } X,\qquad u=f \,\,\mbox{ and } \,\, \pdr{u}{\nu} = j \,\, \mbox{ on } \partial X,
\end{equation}
where $(H,f,j)$ are now known while $u$ is unknown. Thus the measurement operator maps $(\sigma,u)$ to $(H,f,j)$ constructed from $u(x)$ solution of \eqref{eq:elliptic}. Although this problem \eqref{eq:Cauchy1} may look elliptic at first, it is in fact hyperbolic as we already mentioned and this is the reason why we augmented it with full (redundant) Cauchy data. In the sequel, we also consider other  redundant measurements given by the acquisition of $H(x)=\sigma(x)|\nabla u|^2(x)$ for solutions $u$ corresponding to several boundary conditions $f(x)$. A general methodology to uniquely reconstruct $\sigma(x)$ from a sufficient number of redundant measurements has recently been analyzed in \cite{BBMT-11,CFGK-SJIS-09}.

The above equation may be transformed as 
\begin{equation}
  \label{eq:Cauchy2}
  (I-2\widehat{\nabla u}\otimes\widehat{\nabla u}) : \nabla^2 u + \nabla \ln H\cdot\nabla u =0\,\mbox{ in } X,\qquad u=f \,\,\mbox{ and } \,\, \pdr{u}{\nu} = j \,\, \mbox{ on } \partial X.
\end{equation}
Here $\widehat{\nabla u}=\frac{\nabla u}{|\nabla u|}$. With 
\begin{equation}
  \label{eq:gh}
  g^{ij}=g^{ij}(\nabla u)=-\delta^{ij}+2(\widehat{\nabla u})_i(\widehat{\nabla u})_j\quad\mbox{ and } \quad
  k^i=-(\nabla \ln H)_i,
\end{equation}
then \eqref{eq:Cauchy2} is recast as:
\begin{equation}
  \label{eq:Cauchy}
    g^{ij}(\nabla u) \partial^2_{ij} u + k^i \partial_i u=0\,\mbox{ in } X,\qquad u=f \,\,\mbox{ and } \,\, \pdr{u}{\nu} = j \,\, \mbox{ on } \partial X.
\end{equation}
 Note that $g^{ij}$ is a definite matrix of signature $(1,n-1)$ so that \eqref{eq:Cauchy} is a quasilinear strictly hyperbolic equation. The Cauchy data $f$ and $j$ then need to be provided on a space-like hyper-surface in order for the hyperbolic problem to be well-posed \cite{H-II-SP-83}. This is the main difficulty when solving \eqref{eq:Cauchy1} with redundant Cauchy boundary conditions.
%
\section{Local existence, uniqueness, and stability}
\label{sec:local}

Once we recast \eqref{eq:Cauchy1} as the nonlinear hyperbolic equation \eqref{eq:Cauchy}, we have a reasonable framework to perform local reconstructions. However, in general, we cannot hope to reconstruct $u(x)$, and hence $\sigma(x)$ on the whole domain $X$, at least not in a stable manner. The reason is that the direction of ``time" in the second-order hyperbolic equation is $\widehat{\nabla u}(x)$. The normal $\nu(x)$ at the boundary $\partial X$ separates the (good) part of $\partial X$ that is ``space-like" and the (bad) part of $\partial X$ that is ``time-like"; see definitions below. Cauchy data on space-like surfaces such as $t=0$ provide stable information to solve standard wave equations where as in general it is known that arbitrary singularities can form in a wave equation from information on ``time-like" surfaces such as $x=0$ or $y=0$ in a three dimensional setting (where $(t,x,y)$ are local coordinates of $X$) \cite{H-II-SP-83}.

The two-dimensional setting $n=2$ is unique with this respect since the numbers of space-like and time-like variables both equal $1$ and ``t" and ``x" play a symmetric role. Nonetheless, if there exist points at the boundary of $\partial X$ such that $\nu(x)$ is ``light-like" (null), then singularities can form at such points and propagate inside the domain. As a consequence, even in two dimensions of space, instabilities are expected to occur in general. 

We present local results of uniqueness, stability of the reconstruction of $u$ and $\sigma$ in section \ref{sec:localuniq}. These results are based on the linear theory of hyperbolic equations with general Lorentzian metrics \cite{Taylor-PDE-1}. In section \ref{sec:localrec}, we adapt results in \cite{H-SP-97} to propose a local theory of reconstruction of $u(x)$, and hence $\sigma(x)$, by solving \eqref{eq:Cauchy} with data $(H,f,j)$ that are not necessarily in the range of the measurement operator $(u,\sigma)\mapsto (H,f,j)$, which to $(u,\sigma)$ satisfying \eqref{eq:elliptic} associates the Cauchy data $(f,j)$ and the internal functional $H$.

%
\subsection{Uniqueness and stability}
\label{sec:localuniq}

Stability estimates may be obtained as follows. Let $(u,\sigma)$ and $(\tilde u,\tilde\sigma)$ be two solutions of \eqref{eq:elliptic} and the Cauchy problem \eqref{eq:Cauchy} with measurements $(H,f,j)$ and $(\tilde H,\tilde f,\tilde j)$. Note that after solving \eqref{eq:Cauchy}, we then reconstruct the conductivities with
\begin{equation}
  \label{eq:reccond} \sigma(x) = \dfrac{H}{|\nabla u|^2}(x), \qquad \tilde \sigma(x) = \dfrac{\tilde H}{|\nabla \tilde u|^2}(x).
\end{equation}
The objective of stability estimates is to show that $(u-\tilde u,\sigma-\tilde\sigma)$ are controlled by $(H-\tilde H,f-\tilde f,j-\tilde j)$, i.e., to show that small errors in measurements (that are in the range of the measurement operator) correspond to small errors in the coefficients that generated such measurements. 

Some algebra shows that $v=\tilde u- u$ solves the following linear equation
\begin{displaymath} 
 \nabla\cdot \Big( \dfrac{H}{|\nabla \tilde u|^2} \Big\{ I-\dfrac{\nabla u \otimes (\nabla u+\nabla \tilde u)}{|\nabla u|^2} \Big\} \nabla v + \dfrac{H-\tilde H}{|\nabla \tilde u|^2} \nabla \tilde u\Big)=0,
\end{displaymath}
with Cauchy data $\tilde f-f$ and $\tilde j-j$, respectively.  Changing the roles of $u$ and $\tilde u$ and summing the two equalities, we get
 \begin{displaymath} 
 \nabla\cdot \Big( \dfrac{H}{|\nabla \tilde u|^2|\nabla u|^2} \Big\{ (\nabla u+\nabla \tilde u) \otimes (\nabla u+\nabla \tilde u)-(|\nabla u|^2+|\nabla \tilde u|^2)I \Big\} \nabla v + \delta H \Big(\dfrac{\nabla \tilde u}{|\nabla \tilde u|^2}+\dfrac{\nabla u}{|\nabla u|^2}\Big)\Big)=0.
\end{displaymath}
The above operator is elliptic when $\nabla u\cdot\nabla \tilde u<0$ and is hyperbolic when  $\nabla u\cdot\nabla \tilde u>0$. Note that $\nabla u\cdot\nabla\tilde u>0$ on $\partial X$ when $j-\tilde j$ and $f-\tilde f$ are sufficiently small. We obtain a linear equation for $v$ with a source term proportional to $\delta H=\tilde H-H$. For large amounts of noise, $\nabla u$ may significantly depart from $\nabla \tilde u$, in which case the above equation may loose its hyperbolic character. However, stability estimates are useful when $\delta H$ is small, which should imply that $u$ and $\tilde u$ are sufficiently close, in which case the above operator is hyperbolic. We assume here that the solutions $u$ and $\tilde u$ are sufficiently close so that the above equation is hyperbolic throughout the domain. We recast the above equation as the linear equation
\begin{equation}
  \label{eq:linCauchy} 
  \mg^{ij}(x) \partial^2_{ij} v + \mk^i \partial_i v + \partial_i (l^i \delta H)=0 \quad \mbox{ in } X,\qquad
  v= \tilde f-f, \quad \pdr{v}{\nu} = \tilde j-j \quad \mbox{ on } \partial X,
\end{equation}
for appropriate coefficients $\mg^{ij}$, $\mk^i$ and $l^i$.
Now $\mg^{ij}$ is strictly hyperbolic in $X$ (of signature $(1,n-1)$) and is given explicitly by
\begin{equation}\label{eq:mg} 
  \begin{array}{rcl}
  \mg(x) &=&  \dfrac{H}{|\nabla \tilde u|^2|\nabla u|^2} \Big\{ (\nabla u+\nabla \tilde u) \otimes (\nabla u+\nabla \tilde u)-(|\nabla u|^2+|\nabla \tilde u|^2)I \Big\} \\[3mm]
    &=& 
    \alpha(x) \Big(\,\be(x) \otimes \be(x) - \beta^2(x) \big(I-\be(x) \otimes \be(x) \big) \Big),
    \end{array}
 \end{equation} 
where
\begin{equation}\label{eq:bebeta}
    \be(x) = \dfrac{\nabla u+\nabla \tilde u}{|\nabla u+\nabla \tilde u|}(x) ,\qquad \beta^2(x) = \dfrac{|\nabla u|^2+|\nabla \tilde u|^2}{|\nabla u+\nabla \tilde u|^2- (|\nabla u|^2+|\nabla \tilde u|^2)}(x) ,
 \end{equation}
and $\alpha(x)=\frac{H}{|\nabla \tilde u|^2|\nabla u|^2}(|\nabla u+\nabla \tilde u|^2- (|\nabla u|^2+|\nabla \tilde u|^2))$ is the appropriate (scalar) normalization constant. Here, $\be(x)$ is a normal vector that gives the direction of ``time" and $\beta(x)$ should be seen as a speed of propagation (close to $1$ when $u$ and $\tilde u$ are close). When $\be$ is constant, then the above metric, up to normalization, corresponds to the operator $\partial^2_t-\beta^2(t,x')\Delta_{x'}$.

We also define the Lorentzian metric $\mh=\mg^{-1}$ so that $\mh_{ij}$ are the coordinates of the inverse of the matrix $\mg^{ij}$. We denote by $\aver{\cdot,\cdot}$ the bilinear product associated to $\mh$ so that $\aver{u,v}=\mh_{ij}u^iv^j$ where the two vectors $u$ and $v$ have coordinates $u^i$ and $v^i$, respectively. We verify that 
\begin{equation}
  \label{eq:mh}
  \mh(x) \,=\, \frac1{\alpha(x)} \Big(\,\be(x) \otimes \be(x) - \dfrac1{\beta^{2}(x)} \big(I-\be(x) \otimes \be(x) \big) \Big).
\end{equation}
 
The main difficulty in obtaining a solution $v$ to \eqref{eq:linCauchy} arises because $\nu(x)$ is not time-like for all points of $\partial X$. The space-like part $\Sigma_g$ of $\partial X$ is given by the points $x\in\partial X$ such that $\nu(x)$ is time-like, i.e., $\mh(\nu(x),\nu(x))>0$, or equivalently
\begin{equation}
  \label{eq:spacelikepartialX}
  |\nu(x) \cdot \be(x) |^2 >  \frac 1{1+\beta^2(x)} \qquad x\in \partial X.
\end{equation}
Above, the ``dot" product is with respect to the standard Euclidean metric and $\nu$ is a unit vector for the Euclidean metric, not for the metric $\mh$. The time-like part of $\partial X$ is given by the points $x\in \partial X$ such that  $\mh(\nu(x),\nu(x))<0$ (i.e., $\nu(x)$ is a space-like vector) while the light-like (null) part of $\partial X$ corresponds to $x$ such that $\mh(\nu(x),\nu(x))=0$ (i.e., $\nu(x)$ is a null vector).

When $j=\tilde j$ on $\partial X$ so that $\nabla u(x)=\nabla\tilde u(x)$ and $\beta(x)=1$ for $x\in\partial X$  (see also the proof of Theorem \ref{thm:LinearStab} below), then the above constraint becomes:
\begin{equation}
  \label{eq:spacelikepartialX2}
  |\nu(x) \cdot \widehat{\nabla u}(x) |^2 >  \frac 12 \qquad x\in \partial X.
\end{equation}
 In other words, when such a constraint is satisfied, the differential operator is strictly hyperbolic with respect to $\nu(x)$ on $\Sigma_g$. Once $\Sigma_g$ is constructed, we need to define its domain of influence $X_g\subset X$, i.e., the domain in which $v$ can be calculated from knowledge of its Cauchy data on $\Sigma_g$. In order to do so, we apply the energy estimate method for hyperbolic equations described in \cite[Section 2.8]{Taylor-PDE-1}. We need to introduce some notation as in the latter reference; see Figure \ref{fig:localconst}.

\begin{figure}[ht]
\begin{center}
\includegraphics[width=12cm]{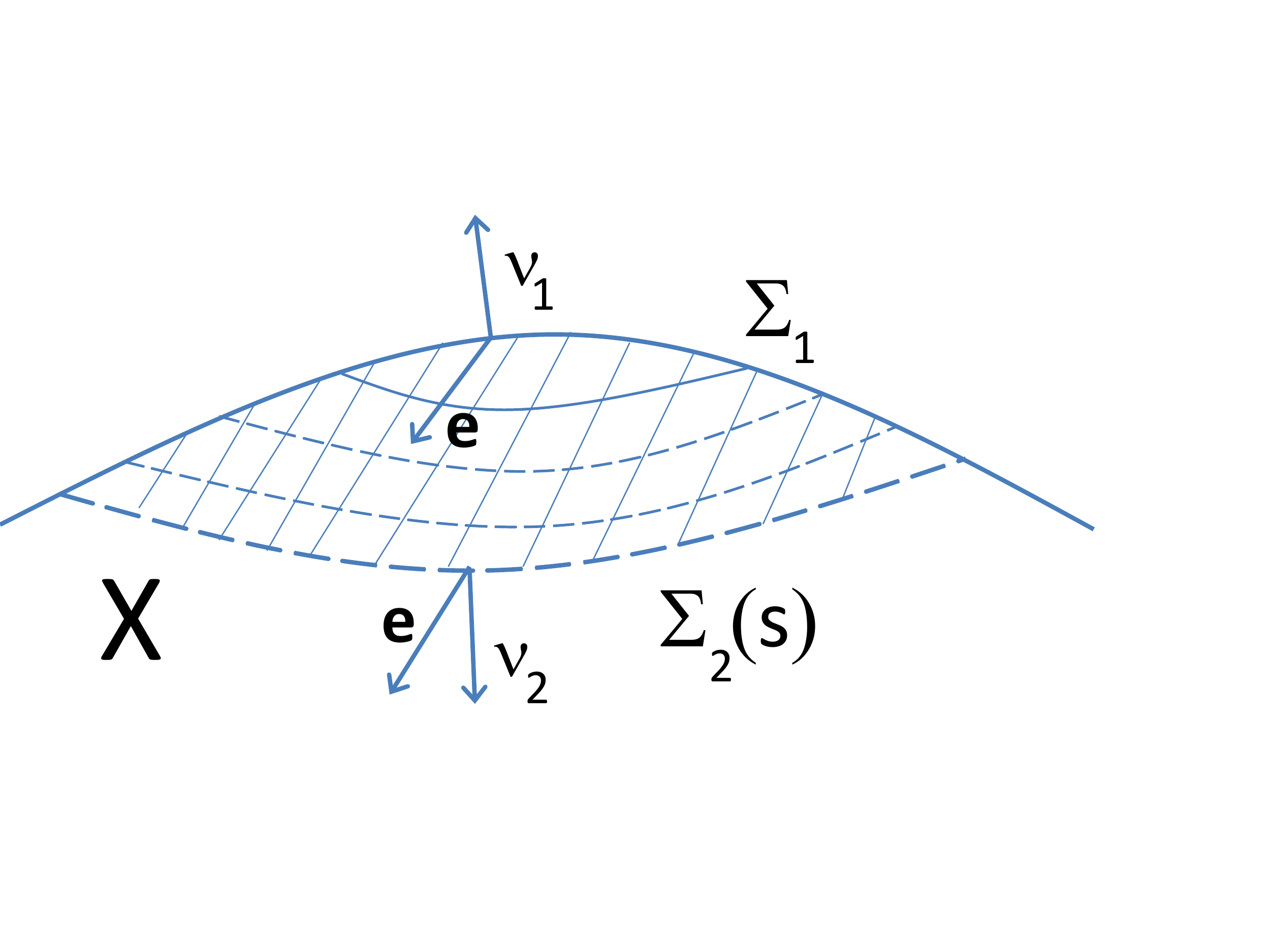}
\end{center}
\caption{Construction of the domain of influence $\cO$ (hatched area in the above picture). The unit vectors $\be$ indicate the ``time" direction of the Lorentzian metric $\mh$. The surface $\Sigma_2(s)$ has a normal vector $\nu_2(x)$ that forms a sufficiently small angle with $\be$ so that $\Sigma_2(s)$ is a space-like surface, as is $\Sigma_1\subset\Sigma_g$ with an angle such that $|\nu_1\cdot\be|$ is also sufficiently close to $1$.}
\label{fig:localconst}
\end{figure}

Let $\Sigma_1$ be an open connected component of $\Sigma_g$. We assume here that all coefficients and geometrical quantities are smooth. By assumption, $\Sigma_1$ is space-like, which means that the normal vector $\nu_1$ is time-like and hence satisfies \eqref{eq:spacelikepartialX}. Let now $\Sigma_2(s)\subset X$ be a family of (open) hyper-surfaces that are also space-like with unit (with respect to the Euclidean metric) vector $\nu_2(x)$ that is thus time-like, i.e., verifies \eqref{eq:spacelikepartialX}. We assume that the boundary of $\Sigma_2(s)$ is a co-dimension 1 manifold of $\Sigma_1$. Let then 
\begin{equation}\label{eq:O} 
   \cO (s) = \bigcup\limits_{0<\tau<s} \Sigma_2(\tau),
\end{equation}
which we assume is an open subset of $X$. In other words, we look at domains of influence $\cO (s)$ of $\Sigma_1$ that are foliated (swept out) by the space-like surfaces $\Sigma_2(\tau)$. Then we have the following result:
\begin{theorem}[Local Uniqueness and Stability.]
  \label{thm:LinearStab}
  Let $u$ and $\tilde u$ be two solutions of \eqref{eq:Cauchy1} sufficiently close in the $W^{1,\infty}(X)$ norm and such that $|\nabla u|$, $|\nabla \tilde u|$, $H$ and $\tilde H$ are bounded above and below by positive constants.  This ensures that $\mg$ constructed in \eqref{eq:mg} is strictly hyperbolic and that $\alpha(x)$ and $\beta(x)$ in \eqref{eq:bebeta} are bounded above and below by positive constants.
  
  Let $\Sigma_1$ be an open connected component of $\Sigma_g$ the space-like component of $\partial X$ and let the domain of influence $\cO=\cO(s)$ for some $s>0$ be constructed as above. Let us define the energy
  \begin{equation}
  \label{eq:energy}
  E(dv) = \aver{dv,\nu_2}^2 - \dfrac12 \aver{dv,dv}\aver{\nu_2,\nu_2}.
\end{equation}
Here, $dv$ is the gradient of $v$ in the metric $\mh$ and is thus given in coordinates by $\mg^{ij}\partial_j v$. Then we have the local stability result
\begin{equation}
  \label{eq:localstab}
  \dint_{\cO} E(dv) dx \leq C \Big( \dint_{\Sigma_1} |f-\tilde f|^2 + |j-\tilde j|^2 \,d\sigma + \dint_{\cO} |\nabla \delta H|^2 \,dx\Big),
\end{equation}
where $dx$ and $d\sigma$ are the standard (Euclidean) volume and (hyper) surface measures on $\cO$ and $\Sigma_1$, respectively.

The above estimate is the natural estimate for the Lorentzian metric $\mh$. For the Euclidean metric, the above estimate may be modified as follows. Let $\nu_2(x)$ be the unit (for the Euclidean metric) vector to $x\in\Sigma_2(s)$ and let us define $c(x) := \nu_2(x) \cdot \be(x)$ with $\be(x)$ as in \eqref{eq:bebeta}.  Let us define
\begin{equation}
  \label{eq:theta0} \theta:= \min_{x\in \cO} \Big[c^2(x) - \dfrac{1}{1+\beta^2(x)}\Big].
\end{equation}
We need $\theta>0$ for the metric $\mh$ to be hyperbolic with respect to $\nu_2(x)$ for all $x\in\cO$.
Then we have that 
\begin{equation}
  \label{eq:localstab2}
  \dint_{\cO} |v^2|+|\nabla v|^2 + (\sigma-\tilde\sigma)^2 \,dx \leq \dfrac{C}{\theta^2} \Big( \dint_{\Sigma_1} |f-\tilde f|^2 + |j-\tilde j|^2 \,d\sigma + \dint_{\cO} |\nabla \delta H|^2 \,dx\Big),
\end{equation}
where $\sigma$ and $\tilde\sigma$ are the reconstructed conductivities given in \eqref{eq:reccond}.
Provided that data are equal in the sense that $f=\tilde f$, $j=\tilde j$, and $H=\tilde H$, we obtain that $v=0$ and the uniqueness result $u=\tilde u$ and $\sigma=\tilde\sigma$.
\end{theorem}
\begin{proof}
That $\mh$ is a hyperbolic metric is obtained for instance if $u$ and $\tilde u$ are sufficiently close in the $W^{1,\infty}(X)$ norm and if $|\nabla u|$, $|\nabla \tilde u|$, $H$ and $\tilde H$ are  bounded above and below by positive constants. The derivation of \eqref{eq:localstab} then follows from \cite[Proposition 8.1]{Taylor-PDE-1} using the notation introduced earlier in this section. The volume and surface measures $dx$ and $d\sigma$ are here the Euclidean measures and are of the same order as the volume and surface measures of the Lorentzian metric $\mh$.  This can be seen in \eqref{eq:mh} since $\alpha$ and $\beta$ are bounded above and below by positive constants. 

Then \eqref{eq:localstab} reflects the fact that the energy measured by the metric $\mh$ is controlled. However, such an ``energy" fails to remain definite for null-like vectors (vectors $\bv$ such that $\mh(\bv,\bv)=0$) and as $x$ approaches the boundary of the domain of influence of $\Sigma_g$, we expect the estimate to deteriorate. 

Let $x\in\cO$ be fixed and define $\nu=\nu_2(x)$ and $\be=\be(x)$. Let us decompose $\nu=c\be+s'\be^\perp$, where $c\be$ is the orthogonal projection of $\nu$ onto $\be$ and $s'\be^\perp:=\nu-c\be$ the projection onto the orthogonal subspace of $\Rm^n$ with $\be^\perp$ a unit vector. 
For a vector $\bv=v_1\be+v'_2\be^\perp+\bw'$ (standing for $dv$) with $\bw'$ orthogonal to $\be$ and $\be^\perp$ (and thus vanishing if $n=2$), we need to estimate
\begin{displaymath} 
   E(\bv) = h^2(\bv,\nu) - \frac12 h(\bv,\bv) h(\nu,\nu)=\frac{1}{\alpha^2}\Big[(v_1c-v_2s)^2-\frac12(v_1^2-(v_2^2+|\bw|^2))(c^2-s^2)\Big],
\end{displaymath}
where we have conveniently defined $v_2=\beta^{-1}v'_2$, $\bw=\beta^{-1}\bw'$, and $s=\beta^{-1}s'$. After some straightforward algebra, we find that
\begin{displaymath} \begin{array}{rcl}
E(\bv)&=&\dfrac{1+\beta^2}{\alpha^2\beta^2} \theta |\bw|^2 + 
\dfrac{1}{\alpha^2}\Big(\dfrac12(c^2+s^2)(v_1^2+v_2^2)-2v_1v_2cs\Big)\\
&\geq & \dfrac{1+\beta^2}{\alpha^2\beta^2} \theta |\bw|^2 + \dfrac{v_1^2+v_2^2}{2\alpha^2}(c-s)^2.
\end{array}
 \end{displaymath}
 Since $\beta$ is bounded above and below by positive constants, we need to bound $(c-s)$ from below or equivalently $(\beta c-s')^2$ from below.
Some algebra shows that 
\begin{displaymath} 
\theta \leq c^2-\frac 1 {1+\beta^2} = \frac{\beta c + s'}{1+\beta^2}(\beta c-s').
 \end{displaymath} 
Since $\theta<1$, this shows that
\begin{displaymath} 
   E(\bv) \geq C \theta^2 |\bv|^2, 
\end{displaymath} 
for a constant $C$ that depends on the lower and upper bounds for $\beta$ and $\alpha$ but not on the geometry of $\nu$. Note that the behavior of the energy in $\theta^2$ is sharp as the bound is attained for $v_1=v_2$ with $\bw=0$. This proves the error estimate for $\bv=\nabla v$. Since $v$ is controlled on $\Sigma_1$, we obtain control of $v$ on $\cO$ by the Poincar\'e inequality. Now $\sigma-\tilde\sigma$ is estimated by $H-\tilde H$ and by $\nabla u-\nabla \tilde u=\nabla v$ and hence the result.
 
 In other words, the angle $\phi(x)$ between $\be(x)$ and $\nu_2(x)$ must be such that $\beta(x)-\tan\phi(x)\geq\theta^2$ in order to obtain a stable reconstruction. When $\delta H$ is small, then  $\nabla u-\tilde \nabla u$ is small so that $\beta$ is close to $1$. As a consequence, we obtain that the constraint of hyperbolicity of $\mh$ is to first order that $\tan\phi(x)<1$, which is indeed the constraint \eqref{eq:spacelikepartialX} that holds when $\nabla u=\nabla\tilde u$ on $\partial X$.

For the uniqueness result, assume that $u$ and $\tilde u$ are two solutions of \eqref{eq:Cauchy1}. We define $\be(x)= \widehat {\nabla u}$ and $\beta^2\equiv1$. Then $v=0$ on $\Sigma_1$ implies by the preceding results that $v=0$ in a vicinity of $\Sigma_1$ in $\cO$ so that $u=\tilde u$ in the vicinity of $\Sigma_1$. This shows that $u=\tilde u$ in $\cO$ and hence in all the domain of dependence of $\Sigma_g$ constructed as above.
\end{proof}

Let us conclude this section with a few remarks. 

\begin{figure}[ht]
\begin{center}
\includegraphics[width=9cm]{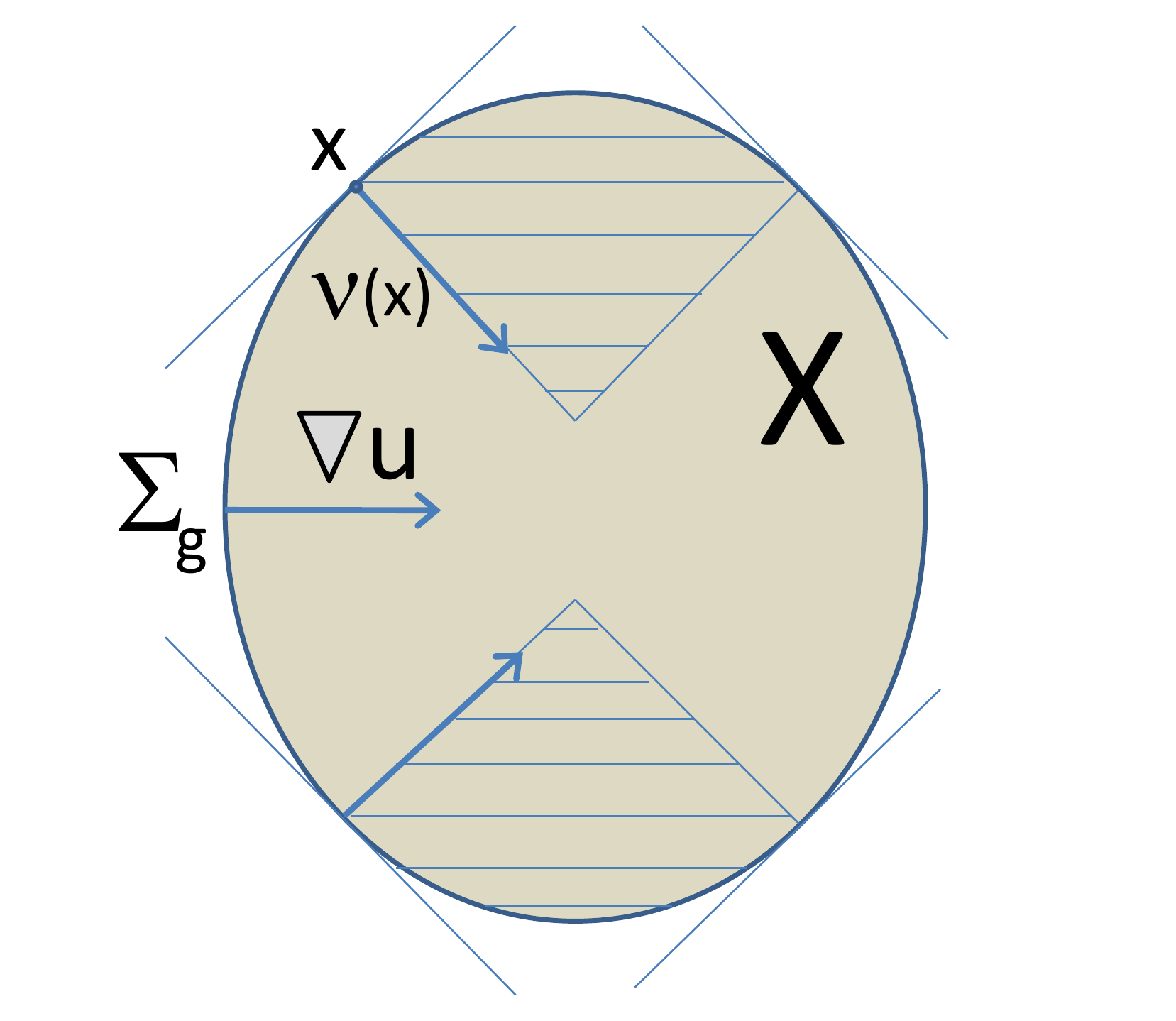}
\end{center}
\caption{Geometry of the domain of influence in Euclidean geometry with $\sigma\equiv1$ and $u=x_1$ on a domain $X$ given by an ovoid. In a three dimensional geometry, we can consider the above picture a cross section at $y=0$ of a three dimensional domain of revolution about the axis $\be_1\equiv \nabla u$. The vector $\nu(x)$ is a ``null-vector" making an angle at $45$ degrees with $\nabla u$.
In two dimensions, $\Sigma_g$ is the union of two connected components where as in thee dimensions, $\Sigma_g$ is composed of a unique connected component in $\partial X$. The hatched domain corresponds to $X\backslash X_g$, the part of the domain $X$ that is not the domain of influence of $\Sigma_g$. In two dimensions, $\nabla u^\perp=\be_2$ may also play the role of ``time" so that $X\backslash X_g$ is the domain of influence of $\Sigma\backslash\Sigma_g$. In three dimension, the hatched region is not accessible with the techniques developed in this paper.}
\label{fig:geom}
\end{figure}

\begin{remark} \rm
In two dimensions, we can exchange the role of space-like and time-like variables since both are one-dimensional and find, at least for sufficiently simple geometries, that the complement of the domain of influence of $\Sigma_g$ in $X$ is the domain of influence of the complement of $\Sigma_g$ in $\partial X$. We thus obtain stability of the reconstruction in all of $X$ except in the vicinity of the geodesics for the metric $\mg$ that emanate from $\partial X$ in a direction $\nu(x)$ that is null-like, i.e., a vector such that $\mh(\nu(x),\nu(x))=0$, or equivalently such that $|\nu(x)\cdot \be(x)|^2=\frac12$.

In three (or higher) dimensions, however, no such exchange of the role of time and space is possible. All we can hope for is a uniqueness and stability result in the domain of influence of $\Sigma_g$. The solution $v$ and the conductivity $\sigma$ are not stably reconstructed on the rest of the domain without additional information, coming e.g. from other boundary conditions $f(x)$. The case of redundant measurements of this type is considered in section \ref{sec:CGO} below, and is analyzed in a different context in \cite{BBMT-11,CFGK-SJIS-09}.
\end{remark}

\begin{remark}\rm Assuming that the errors on the Cauchy data $f$ and $j$ are negligible, we obtain the following stability estimate for the conductivity
\begin{equation}
  \label{eq:stabcond1}
    \|\sigma-\tilde\sigma\|_{L^2(\cO)} \leq \dfrac{C}{\theta} \|H-\tilde H\|_{H^1(X)}.
\end{equation}
The measurements are of the form $H(x)=\sigma(x) |\nabla u|^2(x)$, which imposes reasonably restrictive assumptions on $\sigma$ ensuring that $\nabla u$ is a solution in $H^2(\Omega)$. Under additional regularity assumptions on $\sigma$, for instance assuming that $H\in H^s(X)$ for $s\geq2$, we find that
\begin{equation}
  \label{eq:stabcond2}
    \|\sigma-\tilde\sigma\|_{L^2(\cO)} \leq \dfrac{C}{\theta} \|H-\tilde H\|^{1-\frac 1s}_{L^2(X)} \|H+\tilde H\|^{\frac1 s}_{H^s(X)},
\end{equation}
by standard interpolation. We thus obtain a standard H\"older estimate in the setting where the error in the measurements is quantified in the square integrable sense.
\end{remark}

\begin{remark} \label{rem:linear} \rm
  The linearization of \eqref{eq:Cauchy1} in the vicinity of $\sigma_0=1$ with only Dirichlet data is an ill-posed problem when $X$ is a two-dimensional disc. Indeed, assume Dirichlet data of the form $f(x)=x_1$ in \eqref{eq:elliptic} so that the unperturbed solution is $u_0=x_1$ in $X$. This shows that $\be(x)=\be_1$ in the definition \eqref{eq:mg} so that $\mh=\be_1\otimes \be_1-\be_2\otimes \be_2$ in \eqref{eq:mh}. In other words, the linearized problem consists of solving 
  \begin{displaymath} 
   \pdrr u{x_1} -\pdrr u{x_2} =0 \quad \mbox{ in } X=\{x_1^2+x_2^2<1\},\qquad u=f\quad\mbox{ on } \partial X.
 \end{displaymath}
 The general solution to the above equation is of the form $F(x_1-x_2)+G(x_1+x_2)$ and there is an infinite number of linearly independent solutions to the above equation with $f=0$. The linearization of the UMEIT problem in this specific geometry provides an operator that is not injective.
\end{remark}

%
\subsection{Reconstruction of the conductivity}
\label{sec:localrec}

The construction of the solution $u$, from which we deduce the reconstruction of $\sigma(x)$, requires that we solve the nonlinear equation \eqref{eq:Cauchy}. Let us assume that $g^{ij}$ is given as in \eqref{eq:gh} and that the vector field $h$ and the source terms $f$ and $j$ are smooth given functions. Then we can construct a unique solution to \eqref{eq:Cauchy} locally in the vicinity of the part of $\partial X$ that is space-like. In this section, we assume that the geometry and the coefficients of the wave equation are sufficiently smooth.

Let a point $x^0\in \Sigma_g$, the space-like part of $\partial X$, i.e., so that $g(\nu(x^0),\nu(x^0))\geq\eta>0$. In the vicinity of $x^0$, which we now call $0$, we parameterize $\partial X$ by the variables $(y_1,\ldots,y_{n-1})$ and denote by $y_0$ the signed distance to $\partial X$. In the vicinity of $x^0=0$, $y=F(x)$ is a diffeomorphism from a neighborhood $U$ of  $x=0$ to the neighborhood $V=F(U)$ of $y=0$. Moreover, locally, $DF$ is close to the identity matrix (after appropriate rotation of the domain if necessary) if $U$ is sufficiently small. We denote by $J_F={\rm det}(DF)$ the Jacobian of the transformation.

Let us come back to the equation
\begin{equation}
  \label{eq:Cauchy1b}
  -\nabla \cdot \sigma(x)\nabla u =  -\nabla \cdot \dfrac{H(x)}{|\nabla u|^2(x)} \nabla u =0\,\mbox{ in } X\qquad u=f \,\,\mbox{ and } \,\, \pdr{u}{\nu} = j \,\, \mbox{ on } \partial X.
\end{equation}
We define $v(y)=u(x)$, i.e., $v=F_*u$ and then verify that $(\nabla u)(x)=DF^t\circ F^{-1}(y)\nabla v(y)$. In the $y$ coordinates, we find that 
\begin{displaymath} 
    -\nabla\cdot F_*\sigma \nabla v =0, \quad F(U),
\end{displaymath}
where we have the standard expression in the $y$ coordinates:
\begin{displaymath} 
     F_*\sigma (y) = \tilde\sigma(y) DF DF^t \circ F^{-1}(y),\qquad \tilde \sigma= J_F^{-1} \sigma\,\, \circ F^{-1}.
\end{displaymath}
We may thus recast the above equation as the following non-linear equation
\begin{equation}
  \label{eq:Cauchy2b}
  -\nabla\cdot \tilde H \dfrac{DF DF^t \circ F^{-1}}{|DF^t \nabla v|^2} \nabla v =0,\qquad \tilde H = F_*(J_F^{-1}H) = J_F^{-1} H \,\, \circ F^{-1}.
\end{equation}
Note that the boundary conditions are now posed on the surface $y_0=0$ where
\begin{displaymath} 
     v (0,y') = F_*f (0,y') \qquad \mbox{ and } \qquad \partial_{y_0} v(0,y') = \alpha(y') F_*f(0,y') + \beta(y') F_*j(0,y'),
\end{displaymath}
with $\alpha$ close to $0$ and $\beta$ close to $1$ on $V=F(U)$. It remains to differentiate in \eqref{eq:Cauchy2b} to obtain after straightforward but tedious calculations the expression
\begin{equation}
  \label{eq:Cauchy2t}
  g^{ij}_F \partial^2_{ij} v + h^i_F \partial_i v =0 ,\qquad F(U),
\end{equation}
with the above ``initial" conditions at $y_0=0$, where
\begin{equation}
  \label{eq:gFhF}
   \begin{array}{rcl}
  g^{ij}_F &=& 
  -(DFDF^t)^i_k\delta^{jk} +2 (DF \widehat{DF^t\nabla v})^i (\widehat{DF^t\nabla v})^j,\\
  h^i_F &=& -(\nabla \ln \tilde H \cdot DF DF^t)^i - (\nabla\cdot DF DF^t)^i +2(DF \widehat{DF^t\nabla v})^j (\widehat{DF^t\nabla v})\partial_j ^k DF^i_k.
  \end{array}
\end{equation}
When $F=I$, we recover \eqref{eq:Cauchy2}. The non-linear terms now involve functions of $\widehat{DF^t\nabla v}$.

Note that $g_F=DF \, g \, DF^t$ if we denote $\widehat{DF^t\nabla v}\equiv\widehat{\nabla u}$ and thus transforms as a tensor of type (2,0). As a consequence, the metric (a tensor of type (0,2)) $g_F^{-1}=DF^{-1}g DF^{-t}$ since $g^{-1}=g$ as can be easily verified. Let $\nu_F=F_*\nu=DF \nu \circ F^{-1}$ be the push-forward of the normal vector seen as a vector field.  At $x^0$, the change of variables is such that 
\begin{displaymath} 
   g_{F,F(x^0)}^{-1} (\pdr{}{y_0}, \pdr{}{y_0}) = g_{F,F(x^0)}^{-1}(\nu_F,\nu_F) = g^{-1}_{x^0}(\nu,\nu)=g_{x^0}(\nu,\nu) \geq \eta>0.
 \end{displaymath}
 This shows that $g_F$ remains hyperbolic in the vicinity of $F(x^0)$ since $DFDF^t=I$ at $y=0$ by construction and $DF$ is smooth. Moreover, the above is equivalent to 
 \begin{displaymath} 
     g^{ij}_F \partial^2_{ij} = \square + \gamma^{ij} \partial^2_{ij}, \qquad
       \square = \frac{\partial^2}{\partial y^0\partial y^0} - \Delta_{y'},\quad \Delta_{y'} =     
       \sum_{j=1}^{n-1} \frac{\partial^2}{\partial y^j \partial y^j},
 \end{displaymath}
 where $\sum_{i,j} |\gamma^{ij}|\leq \frac{1-\eta}2$.  The above is nothing but the fact that $g_F$ is hyperbolic in the vicinity of $y=0$. Note that $\gamma^{ij}=\gamma^{ij}(x,\widehat{DF^t\nabla v}$).
 
We thus have a nonlinear hyperbolic equation of the form
\begin{equation}
  \label{eq:Cauchy0}
  \begin{array}{l}
   \big( \square + \gamma^{ij}(x,\widehat{DF^t\nabla v})\partial_{ij} + h^i(x,\widehat{DF^t\nabla v}) \partial_i \big) v =0, \qquad y_0>0,\,\, y'\in\Rm^{n-1} \\
   v(0,y') = v_0(y'), \qquad \partial_{y_0} v(0,y') = j_0(y').
\end{array}
\end{equation} 
Since propagation in a wave equation is local we can extend the boundary conditions for $y=(0,y')$ outside the domain $F(U)$ by $v_0=0$ and $\partial_{y_0} v=1$ and the functions $\gamma^{ij}$ and $h^i$ by $0$ outside of $F(U)$. This allows us to obtain an equation posed on the half space $y_0>0$.

The nonlinear functions $\gamma^{ij}(x,\widehat{DF^t\nabla v})$ and $h^i(x,\widehat{DF^t\nabla v})$ are smooth functions of $\nabla v$ except at the points where $\nabla v=0$. However, we are interested in solutions such as $\nabla v$ does not reach $0$ to preserve the hyperbolic structure of $g^{ij}$. Note that $|\nabla v|$ is bounded from below by a positive constant on $y_0=0$ by assumption. We obtain a bound on the uniform norm of the Hessian of $v$, which implies that at least for a sufficiently small interval $y_0\in (0,t_0)$, $|\nabla v|$ does not vanish and $\gamma^{ij}$ and $h^i$ can then be considered as smooth functions of $x$ and $\nabla v$.

Using Theorem 6.4.11 in \cite{H-SP-97} and the remark following (6.4.24) of that reference, the above equation satisfies the hypotheses to obtain an a priori estimate for 
\begin{displaymath} 
        M(y_0) = \dsum_{|\alpha|\leq \kappa+2} \|\partial^\alpha u(y_0,\cdot)\|_{L^2(\Rm^{n-1})},
 \end{displaymath}
with $\kappa$ the smallest integer strictly greater than $\frac {n-1}2$. By Sobolev imbedding, this implies that the second derivatives of $v$ are uniformly bounded so that for at least a small interval, $|\nabla v|$ is bounded away from $0$. 

Once $v$, and hence $u$ is reconstructed, at least in the vicinity of the part $\Sigma_g$ of $\partial X$ that is space like for $\nabla u$, we deduce that 
\begin{displaymath} 
   \sigma(x) = \dfrac{H(x)}{|\nabla u|^2(x)}.
\end{displaymath}
Note that $\nabla u$ cannot vanish by construction so that the above equality for $\sigma(x)$ is well-defined. We already know that a solution to the above nonlinear equation exists in the absence of noise since we have constructed it by solving the original linear equation. In the presence of significant noise, the nonlinear equation may behave in a quite different manner than that for the exact solution. However, the above construction shows that the nonlinear equation can be solved locally if the measurement $H(x)$ is perturbed by a small amount of noise.

%
\section{Global reconstructions of the diffusion coefficient}
\label{sec:global}
%

The picture in Fig. \ref{fig:geom} shows that in general, we cannot hope to obtain a global reconstruction from a single measurement of $H(x)$ even augmented with full Cauchy data. Only Cauchy data on the space-like part of the  boundary can be used to obtain stable reconstructions. 

Global reconstructions have been obtained from redundant measurements of the form $H_{ij}=S_i\cdot S_j$ with $S_i=\sqrt\sigma \nabla u_i$ and $u_i$ solution of \eqref{eq:elliptic} with Dirichlet conditions $f=f_i$,  in \cite{CFGK-SJIS-09} in the two dimensional setting and in \cite{BBMT-11} in the two- and three- dimensional settings; see also \cite{KK-AET-11}.


This section analyzes geometries in which a unique measurement $H(x)$ or a small number of measurements of the form $H(x)$, augmented with Cauchy data $(f,j)$ allow one to uniquely and stably reconstruct $\sigma(x)$ on the whole domain $X$. These reconstructions are obtained by (possibly) modifying the geometry of the problem so that the domain where $\sigma(x)$ is not known lies within the domain of dependence of $\Sigma_g$. We consider two scenarios. In the first scenario, considered in section \ref{sec:internalsource}, we slightly modify the problem to obtain a model with an internal source of radiation $f$. Such geometries are guaranteed to provide a unique global reconstruction in dimension $n=2$ but not necessarily in higher spatial dimensions, where global reconstructions hold only for a certain class of coefficients $\sigma(x)$. In the second scenario, analyzed in section \ref{sec:CGO}, we consider a setting where reconstructions are possible when the Lorentzian metric is the Euclidean (Lorentzian) metric, i.e., $\alpha=\beta=1$ in \eqref{eq:mh}. We then show the existence of an open set of illuminations $f$ for three different measurements of the form $H(x)$ such that the global result obtained for the Euclidean metric remains valid for arbitrary, sufficiently smooth coefficients $\sigma(x)$.

%
\subsection{Geometries with an internal source}
\label{sec:internalsource}

From the geometric point of view, the Cauchy data are sufficient to allow for full reconstructions when $\Sigma_g=\partial X$, so that the whole boundary $\partial X$ is space-like for the metric $g$, and $X$ is the domain of dependence of $\Sigma_g$. This can happen for instance when $\partial X$ is a level set of $u$ and the normal derivative of $u$ either points inwards or outwards at every point of $\partial X$. When $X$ is a simply connected domain, the maximum principle prevents one from having such a geometry. However, when $X$ is not simply connected, then such a configuration can arise. We will show that such a configuration (with $X$ the domain of dependence of $\Sigma_g$) is always possible in two dimensions of space. When $n\geq3$, such configurations hold only for a restricted class of conductivities $\sigma(x)$ for which no critical points of $u(x)$ exist.

Let us consider the two dimensional case $n=2$. We assume that $X$ is an open smooth domain diffeomorphic to an annulus and with boundary $\partial X=\partial X_0\cup\partial X_1$; see Fig. \ref{fig:geomannulus}. We assume that $f=0$ on the external boundary $\partial X_0$ and $f=1$ on the internal boundary $\partial X_1$. The boundary of $X$ is composed of two smooth connected components that are different level sets of the solution $u$ to \eqref{eq:elliptic}, which is uniquely defined in $X$. 
\begin{figure}[ht]
\begin{center}
 \includegraphics[width=9cm]{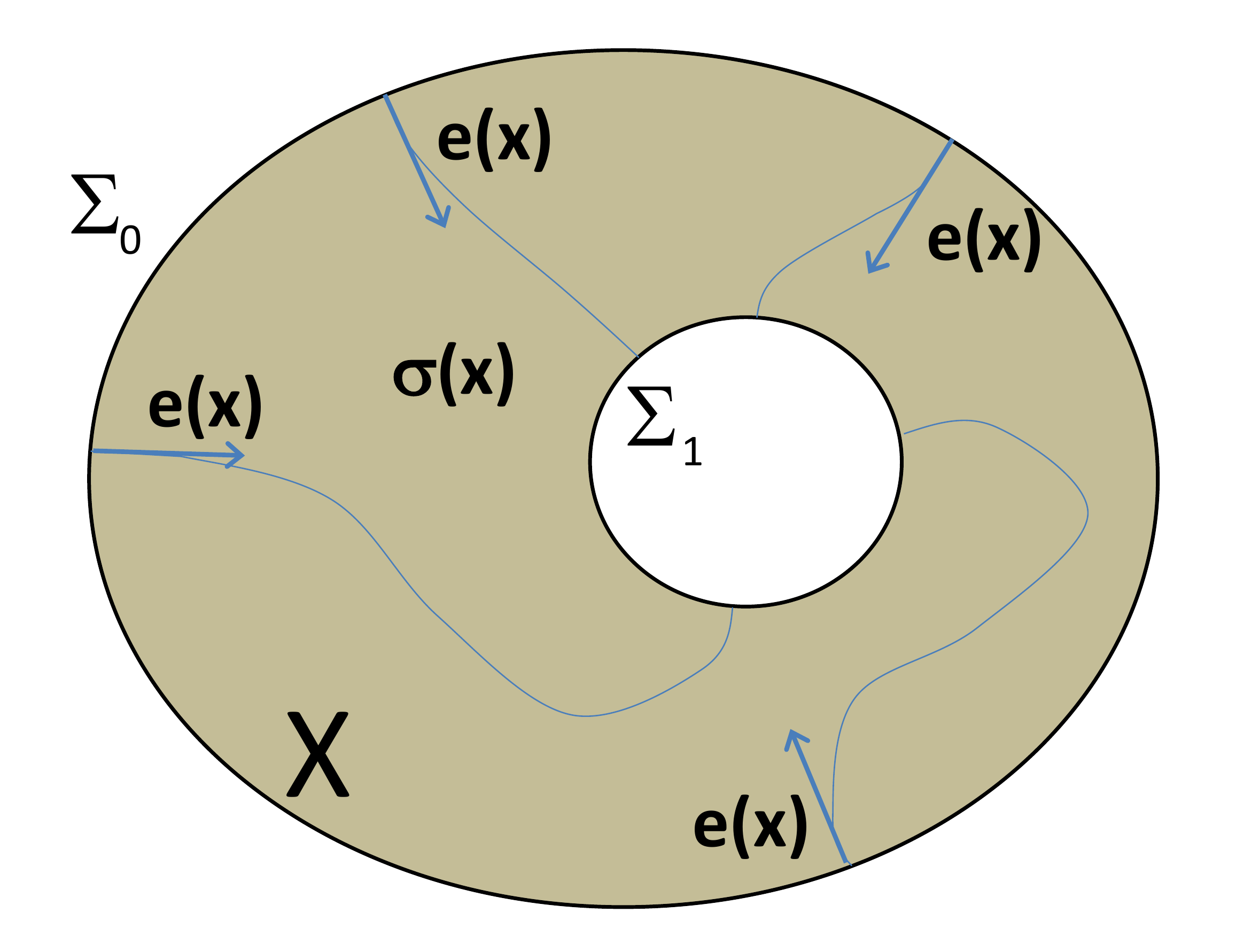}
\end{center}
\caption{Geometry of an annulus in two space dimensions with boundaries $\Sigma_0$ and $\Sigma_1$, the level sets where $u=0$ and $u=1$, respectively. The four curves correspond to four integrals of the flow of the gradient vector field $\widehat{\nabla u}$.}
\label{fig:geomannulus}
\end{figure}

In practice, such a domain $X$ may be constructed as follows. As we do in the geometry depicted in Fig. \ref{fig:geomextension} below, we embed $\tilde X$, the domain where $\sigma$ is unknown, into a larger domain $X$ with, e.g., $\sigma(x)=\sigma_0$ on $X\backslash\tilde X$  and with a hole where we impose the aforementioned boundary conditions.
Then we have the following result:
\begin{proposition}
\label{prop:doughnut2d}
 Let $X$ be the geometry described above with $n=2$ and $u(x)$ the solution to \eqref{eq:elliptic}. We assume here that both the geometry and $\sigma(x)$ are sufficiently smooth. Then $|\nabla u|$ is bounded from above and below by positive constants. The level sets $\Sigma_c=\{x\in X,\, u(x)=c\}$ for $0<c<1$ are smooth curves that separate $X$ into two disjoint subdomains.
\end{proposition}
\begin{proof}
     The proof of the first part is based on the fact that critical points of solutions to elliptic equations in two dimensions are isolated \cite{A-AMPA-86}. First of all, the Hopf lemma \cite{evans} ensures that no critical point exists on the smooth closed curves $\Sigma_0$ and $\Sigma_1$. Let $x_i$ be the finite number of points where $\nabla u(x_i)=0$. At each $x_i$, the level set of $u$ with value $0<c_i=u(x_i)<1$ is locally represented by $n_i$ ($n_i$ even) smooth simple arcs emanating from $x_i$ that make an angle equal to $2\pi/n_i$ at $x_i$ \cite{A-AMPA-86}. For instance, if only two simple arcs emanate from $x_0$, then these two arcs form a continuously differentiable curve in the vicinity of  $x_0$. Between critical points, level sets of $u$ are smooth by the inverse function theorem. 
     
Let us assume that there is a point $x_i$ with more than two simple arcs leaving  $x_i$.  Let $\gamma_j$, $1\leq j\leq 4$, be such arcs. If $\gamma_1$ meets another critical point, we pick one of the possible other arcs emanating from this critical point to continue the curve $\gamma_1$. This is always possible as critical points always have an even number of leaving simple arcs. The curve $\gamma_1$ cannot meet $\Sigma_0$ or $\Sigma_1$ and therefore must come back to the point $x_i$. Let us assume the existence of a closed sub-loop of $\gamma_1$ that does not self-intersect and does not wind around $\Sigma_1$ (i.e., is homotopic to a point). In the interior of that close sub-loop, $u$ is then constant by the maximum principle and hence constant on $X$ by the unique continuation theorem \cite{H-I-SP-83}. This is impossible and therefore $\gamma_1$ must wind around $\Sigma_1$. Let us pick a subset of $\gamma_1$, which we still call $\gamma_1$ that winds around $\Sigma_1$ once. The loop meets one of the other $\gamma_j$ to come back to $x_i$, which we call $\gamma_2$ if it is not $\gamma_1$. Now let us follow $\gamma_3$. Such a curve also has to come back to $x_i$.  By the maximum principle and the unique continuation theorem, it cannot come back with a sub-loop homotopic to a point. So it must come back also winding around $\Sigma_1$. But $\gamma_1$ and $\gamma_3$ are then two different curves winding around $\Sigma_1$. This implies the existence of a connected (not necessarily simply connected) domain whose boundary is included in $\gamma_1\cup\gamma_3$. Again, by the maximum principle and the unique continuation theorem, such a domain cannot exist. So any critical point cannot have more than two simple arcs of level curves of $u$ leaving it.

 So far, we have proved that any critical point $x_i$ sees exactly two arcs leaving $x_i$ at an angle equal to $\pi$ since by the maximum principle, critical points can not be local mimina or maxima. These two arcs again have to meet winding around $\Sigma_1$. This generates what we call a single curve $\gamma_1$ with no possible self-intersection. Moreover, since all angles at critical points are equal to $\pi$, the curve $\gamma_1$ is of class $C^1$ and piecewise of class $C^2$. Let $X_c$ be the annulus with boundary equal to $\Sigma_1\cup\gamma_1$. On $X_c$, $u$ satisfies an elliptic equation with values $u=1$ on $\Sigma_1$ and $0<u=c_i<1$ on $\gamma_1$. Since $\gamma_1$ is sufficiently smooth now (smooth on each arc with matching derivatives on each side of each critical point), it satisfies the  interior sphere condition  and we can apply the Hopf lemma \cite[Lemma 3.4]{gt1} to deduce that the normal derivative of $u$ on $\gamma_1$ cannot vanish at $x_i$ or anywhere along $\gamma_1$. There are therefore no critical points of $u$ in $\bar X$. By continuity, this means that $|\nabla u|$ is uniformly bounded from below by a positive constant. Standard regularity results show that it is also bounded from above.
 
Now let $0<c<1$ and $\Sigma_c$ be the level set where $u=c$. Such a level set separates $X$ into two subdomains where $0<u<c$ and $c<u<1$, respectively, by the maximum principle. We therefore obtain a foliation of $X$ into the union of the smooth curves $\Sigma_c$ for $0<c<1$. Now let $x\in\Sigma_c$ and consider the flow of $\nabla u$ in both directions emanating from $x$. Then both curves are smooth and need to reach the boundary at a unique point. Since any point on $\Sigma_0$ is also mapped to a point on $\Sigma_1$ by the same flow, this shows that $\Sigma_c$ is diffeomorphic to $\Sigma_0$ and $\Sigma_1$.
\end{proof}

The result extends to higher dimensions provided that $|\nabla u|$ does not vanish with exactly the same proof. Only the proof of the absence of critical points of $u$ was purely two-dimensional. In the absence of critical points, we thus obtain that $\be(x)=\widehat{\nabla u}=\nu(x)$ so that $\nu(x)$ is clearly a time-like vector. Then the local results of Theorem \ref{thm:LinearStab} become global results, which yields the following proposition:
\begin{proposition}
\label{prop:doughnutnd}
 Let $X$ be the geometry described above in dimension $n\geq2$ and $u(x)$ the solution to \eqref{eq:elliptic}. We assume here that both the geometry and $\sigma(x)$ are sufficiently smooth. We also {\em assume} that  $|\nabla u|$ is bounded from above and below by positive constants. Then the nonlinear equation \eqref{eq:Cauchy} admits a unique solution and the reconstruction of $u$ and of $\sigma$ is stable in $X$ in the sense  described in Theorem \ref{thm:LinearStab}.
\end{proposition}

\begin{remark}
  \label{rem:doughnut}\rm
The above geometry with a hole is not entirely necessary in practice. Formally, we can assume that the hole with boundary $\Sigma_1$ shrinks and converges to a point $x_0\in\partial X$ at the boundary of the domain. Thus, the illumination $f$ is an approximation of a delta function at $x_0$. The level sets of the solution are qualitatively similar to the level sets in the annulus. Away from $x_0$, the surface $\partial X$ is a level set of the solution $u$ and hence the normal to the level set is a time-like vector for the Lorentzian metric with direction $\be(x)=\nu(x)$. Away from $x_0$, we can solve the wave equation inwards and obtain stable reconstructions in all of $X$ but a small neighborhood of $x_0$. This construction should also provide stable reconstructions in arbitrary dimensions provided that $u$ does not have any critical point.
\end{remark}

In dimensions $n\geq3$, however, we cannot guaranty that $u$ does not have any critical point independent of the conductivity. If the conductivity is close to a constant where we know that no critical point exists, then by continuity of $u$ with respect to small changes in $\sigma(x)$, $u$ does not have any critical point and the above result applies. In the general case, however, we cannot guaranty that $\nabla u$ does not vanish and in fact can produce a counter-example using the geometry introduced in \cite{BMN-ARMA-04} (see also \cite{M-PAMS-93} for the existence of critical points of elliptic solutions):
\begin{proposition}
\label{prop:badsigma3d}
   There is an example of a smooth conductivity such that $u$ admits critical points.
\end{proposition}
\begin{proof}
Consider the geometry in three dimensions depicted in Fig. \ref{fig:critical}.  The domain $X$ is a smooth, convex, domain invariant by rotation leaving $\be_z$ invariant and by symmetry $z\to -z$ and including two disjoint, interlocked, tori $T_1$ and $T_2$. The first torus $T_1$ is centered at $c_1=(0,0,1)$ with base circle $\{\be_z+2\be_x+\alpha(\cos\phi\be_x+\sin\phi\be_y),\, 0\leq \phi<2\pi\}$ rotating around $c_1$ in the plane $(\be_x,\be_z)$ (top torus in Fig. \ref{fig:critical}) for $\alpha=\frac12$, say. The second torus $T_2$ is centered at $c_2=(0,0,-1)$ with base circle $\{-\be_z+2\be_y+\alpha(\cos\phi\be_x+\sin\phi\be_y),\, 0\leq \phi<2\pi\}$ rotating around $c_2$ in the plane $(\be_y,\be_z)$ (bottom torus in Fig. \ref{fig:critical}). 
 \begin{figure}[ht]
\begin{center}
 \includegraphics[width=10cm]{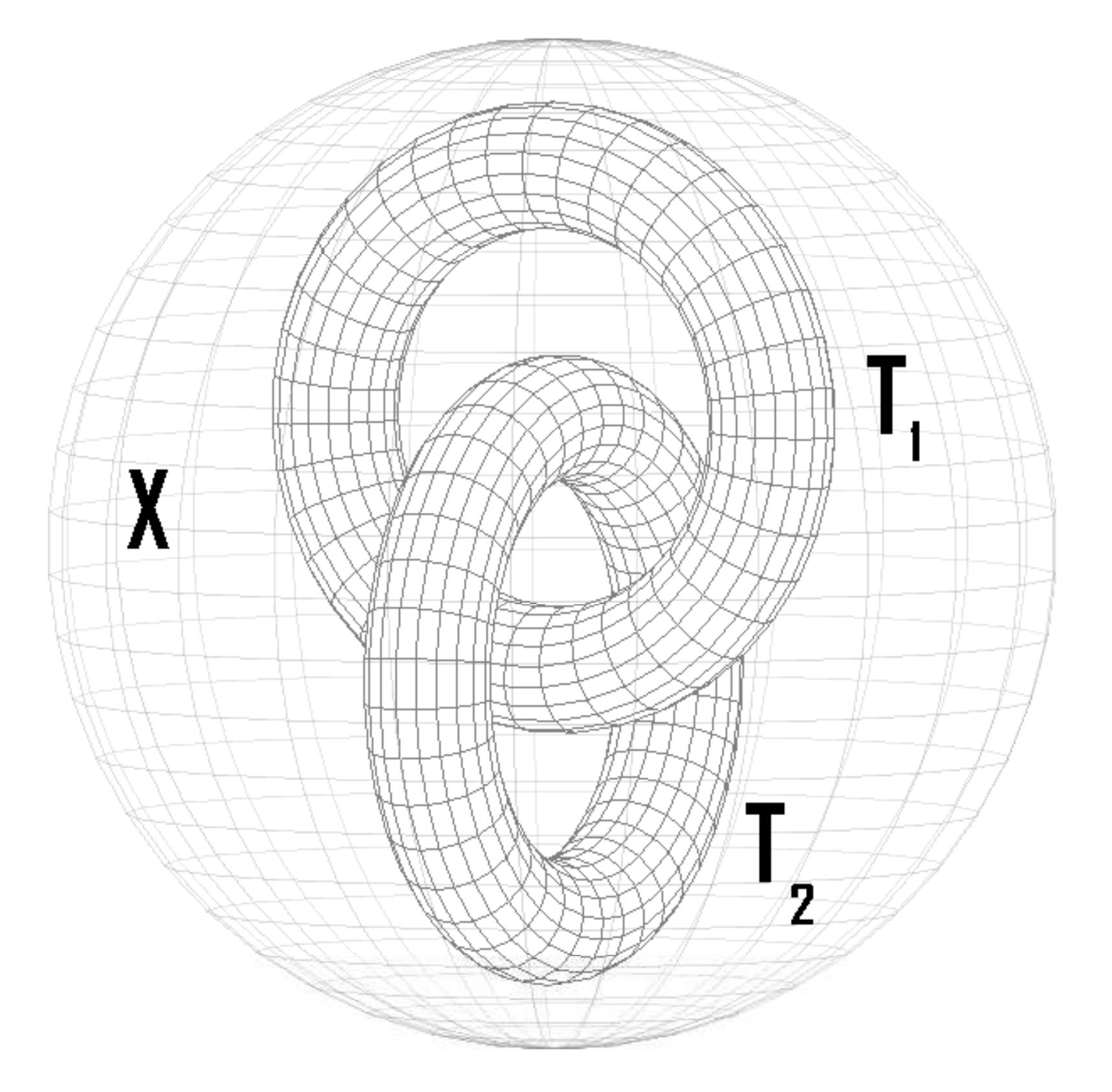}
\end{center}
\caption{Geometry of a critical point. $X$ is the ball of radius $4$. The interlocked tori are the top torus $T_1$ and the bottom torus $T_2$.}
\label{fig:critical}
\end{figure}
We consider the boundary condition $u=z$ on $\partial X$. 

We assume that $\sigma(x)=1+\lambda\varphi(x)$ in \eqref{eq:elliptic}, where $\varphi(x)$ is a smooth, non-trivial, non-negative function with non-vanishing support inside each of the tori $T_1$ and $T_2$ that respects the invariance by rotation and the symmetries of the two tori. 
We normalize $\varphi(x)$ by $1$ on the circles $\{\be_z+2(\cos\phi \be_z+\sin\phi\be_y), \, 0\leq \phi<2\pi\}$ and $\{-\be_z+2(\cos\phi \be_z+\sin\phi\be_y), \, 0\leq \phi<2\pi\}$ at the center of the volumes delimited by the two tori. When $\lambda=0$ so that $\sigma(x)\equiv1$, then $u=z$ is the solution of the problem \eqref{eq:elliptic}.
As $\lambda$, an hence $\sigma$ inside the tori, converges to $+\infty$, the solution $u$ is such that $u$ converges to a constant $C_1>0$ on the support of $\varphi$ inside $T_1$ and $C_2<0$ on the support of $\varphi$ inside $T_2$. For $\lambda$ sufficiently large, by continuity of the solution $u$ with respect to $\sigma$, we obtain that $u(0,0,1)<0$ since $(0,0,1)$ is inside  $T_2$ and $u(0,0,-1)>0$ since $(0,0,-1)$ is inside $T_1$. 
Since the geometry is invariant by symmetry $x\to -x$ and $y\to-y$, then so is the solution $x$ and hence $\partial_x u(0,0,z)=\partial_y u(0,0,z)=0$ for all $(0,0,z)\in X$. Now the function $z\to u(0,0,z)$ goes from negative to positive to negative back to positive values as $z$ increases, and so has at least two critical points. At these points, $\nabla u=0$ and hence the possible presence of critical points in elliptic equations in dimensions three and higher.
\end{proof}

Note that the above symmetries are not necessary to obtain critical points, which appear generically in structures of the form of two interlocked rings with high conductivities as indicated above. We describe this results at a very intuitive, informal, level. Indeed, small perturbations of the above geometry and boundary conditions  make that the level sets $\Sigma_c=\{u=c\}$ for $c$ sufficiently large and $c$ sufficiently small are simply connected co-dimension 1 manifolds with boundary on $\partial X$. When $\sigma$ is sufficiently large, then $u$ converges to two different values $c_1$ and $c_2$ inside the two discs (say one positive in $T_1$ and one negative in $T_2$).  Thus for $\sigma$ sufficiently large, the level set $u=c_1$, assuming it does not have any critical point, is a smooth locally co-dimension 1 manifold, by the implicit function theorem, that can no longer be simply connected. Thus as the level sets $c$ decrease from high values to $c_1$, they go through a change of topology that can only occur at a critical point of $u$ \cite{MC-Els-69}. 

\subsection{Complex geometric optics solutions and global stability}
\label{sec:CGO}

Let us now consider a domain $\tilde X$ where $\sigma(x)$ is unknown and close to a constant $\sigma_0$. Let us assume that $\tilde X$ is embedded into a larger domain $X$ and that we can assume that $\sigma(x)$ is known and also close to the constant $\sigma_0$. Then, it is not difficult to construct $X$ so that $\tilde X$ lies entirely within the domain of dependence of $\Sigma_g$; see for instance the geometry depicted in Fig. \ref{fig:geomextension}. 
\begin{figure}[ht]
\begin{center}
 \includegraphics[width=10cm]{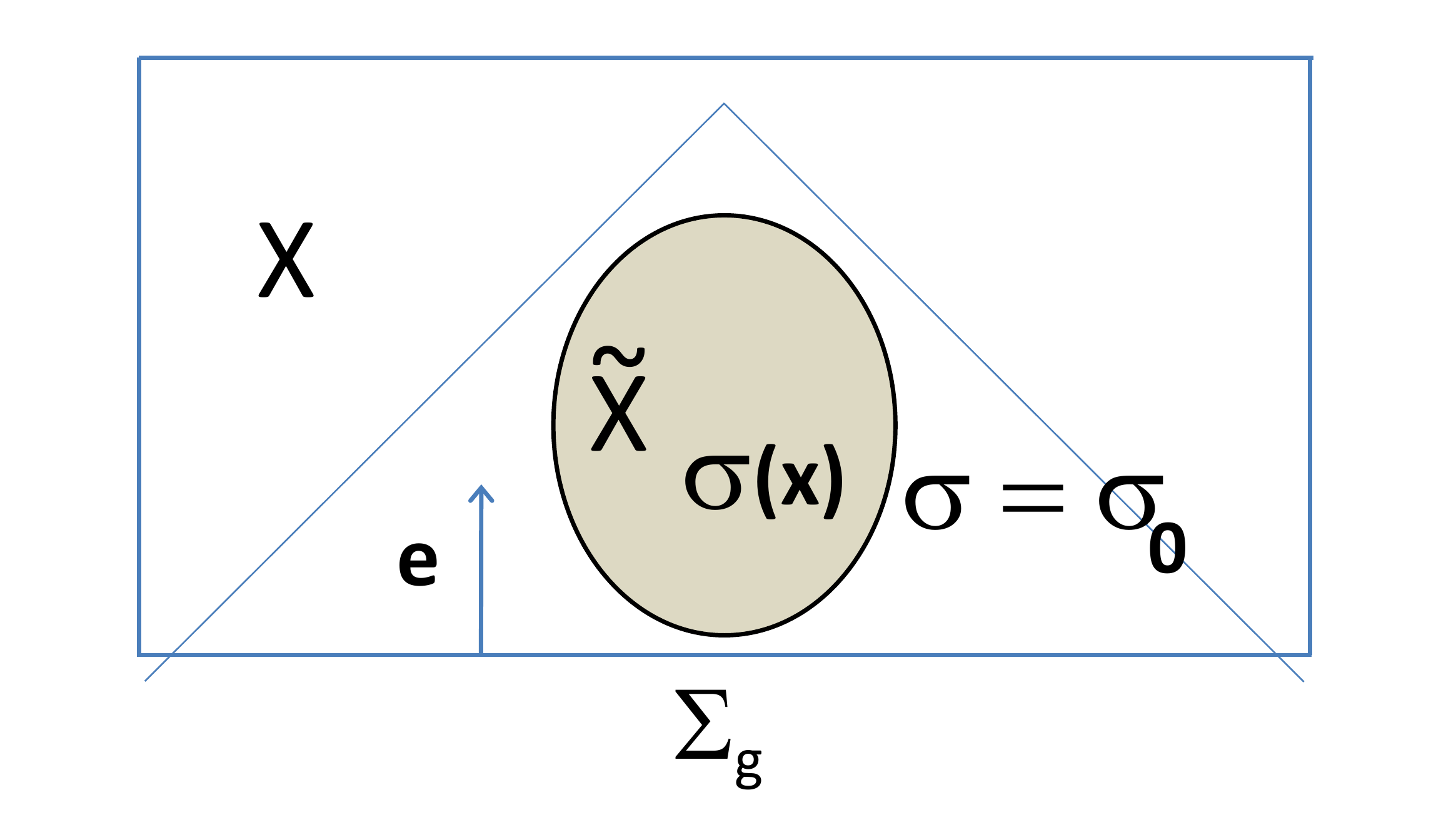}
\end{center}
\caption{Geometry of a domain where the reconstruction of the unknown $\sigma$ on $\tilde X$ is possible from a single measurement. The geometry of the Lorentzian metric is represented when $\sigma(x)=\sigma_0$. By continuity, the domain of influence of $\Sigma_g$ includes $\tilde X$ for all smooth conductivities $\sigma(x)$ sufficiently close to $\sigma_0$.}
\label{fig:geomextension}
\end{figure}

For the rest of the section, we show that global reconstructions can be obtained for general sufficiently smooth metrics provided that three well-chosen measurements are available. This result is independent of spatial dimension. The measurements are constructed by means of complex geometrical optics solutions. 

Let $k$ be a vector in $\Rm^n$ and $k^\perp$ be a vector orthogonal to $k$ of same length. Let $\rho=ik+k^\perp$ be a complex valued vector so that $\rho\cdot\rho=0$. Thus, $e^{\rho\cdot x}$ is harmonic and $\nabla e^{\rho\cdot x}=\rho e^{\rho\cdot x}$. The latter gradient has a privileged direction of propagation $\rho$, which is, however, complex valued. Its real and imaginary parts are such that 
\begin{equation} \label{eq:theta} 
   e^{-k^\perp\!\cdot x}\Im \nabla e^{\rho\cdot x} = |k|\theta(x),\qquad e^{-k^\perp\!\cdot x}\Re \nabla e^{\rho\cdot x}=|k|\theta^\perp(x),
 \end{equation}
where $\theta(x) = \hat k\cos k\!\cdot\! x + \hat k^\perp \sin k\!\cdot \!x$ and $\theta^\perp(x)=-\hat k\sin k\!\cdot \!x+\hat k^\perp\cos k\!\cdot\! x$. As usual, $\hat k=\frac{k}{|k|}$.

Consider propagation with Cauchy data given on a hyperplane with normal vector $\hat k\in\Sm^{n-1}$. We want to make sure that we always have at our disposal a Lorentzian metric for which $\hat k$ is a time-like vector so that the available Cauchy data live on a space-like surface for that metric. For the rest of the section, we assume that $k=|k|\be_1$ and that $k^\perp=|k|\be_2$ so that 
\begin{equation}
\label{eq:spectheta}
\theta(x)=\hat k\cos |k|x_1 + \hat k^\perp \sin |k|x_1\quad\mbox{ and }\quad \theta^\perp(x)=-\hat k\sin |k|x_1+\hat k^\perp\cos |k|x_1
\end{equation}
For a vector field with unit vector $\theta(x)$, we associate the Lorentz metric with direction $\theta$ given by $\mh_\theta=2\theta\otimes\theta-I$. 

The Lorentzian metrics with directions $\theta(x)$ and $\theta^\perp(x)$ oscillate with $x_1$. A given vector $\hat k$ therefore cannot be time-like for all points $x$. However, we can always construct two different linear combinations of these two directions that form time-like vectors for a given range of $k\cdot x=|k|x_1$. Such combinations allow us to solve the wave equation forward and obtain unique and stable reconstructions on the whole domain $X$. The above construction with $e^{\rho\cdot x}$ harmonic can be applied when $\sigma(x)=\sigma_0$ a constant. It turns out that we can construct complex geometric optics solutions for arbitrary, sufficiently smooth conductivities $\sigma(x)$ and obtain global existence and uniqueness results in that setting.  We state the following result. 
\begin{theorem}
  \label{thm:globalCGO}
  Let $\sigma$ be extended by $\sigma_0=1$ on $\Rm^n\backslash \tilde X$, where $\tilde X$ is the domain where $\sigma$ is not known. We assume that $\sigma$ is smooth on $\Rm^n$. Let $\sigma(x)-1$ be supported without loss of generality on the cube $(0,1)\times(-\frac12,\frac12)^{n-1}$. Define the domain $X=(0,1)\times B_{n-1}(a)$, where $B_{n-1}(a)$ is the $n-1$-dimensional ball of radius $a$ centered at $0$ and where $a$ is sufficiently large that the light cone for the Euclidean metric emerging from $B_{n-1}(a)$ strictly includes $\tilde X$. Then there is an open set of illuminations $(f_1,f_2)$ such that if $u_1$ and $u_2$ are the corresponding solutions of \eqref{eq:elliptic}, then the following measurements
  \begin{equation}
  \label{eq:threemeas} H_{11}(x)=\sigma(x) |\nabla u_1|^2(x),\,\,\, H_{22}(x)=\sigma(x)|\nabla u_2|^2(x),\,\,\,
  H_{12}(x)= \sigma(x)|\nabla (u_1+u_2)|^2\!,
\end{equation}
with the corresponding Cauchy data $(f_1,j_1)$, $(f_2,j_2)$ and $(f_1+f_2,j_1+j_2)$ at $x_1=0$ uniquely determine $\sigma(x)$. Moreover, let $\tilde H_{ij}$ be measurements corresponding to $\tilde\sigma$ and $(\tilde f_1,\tilde j_1)$ and $(\tilde f_2,\tilde j_2)$ the corresponding Cauchy data at $x_1=0$. We assume that $\sigma(x)-1$ and $\tilde\sigma(x)-1$ (also supported in $(0,1)\times(-\frac12,\frac12)^{n-1}$) are smooth and such that their norm in $H^{\frac n2+3+\eps}(\Rm^n)$ for some $\eps>0$ are bounded by $M$. Then for a constant $C$ that depends on $M$, we have the global stability result
\begin{equation}
  \label{eq:globalstab}
  \|\sigma-\tilde\sigma\|_{L^2(\tilde X)}\leq C \Big(\|d_C-\tilde d_C\|_{(L^2(B_{n-1}(a)))^4} + \sum_{(i,j)\in I} \|\nabla H_{ij}-\nabla \tilde H_{ij}\|_{L^2(X)}\Big).
\end{equation}
Here, we have defined $I=\{(1,1),(1,2),(2,2)\}$ and $d_C=(f_1,j_1,f_2,j_2)$ with $\tilde d_C$ being defined similarly.
  \end{theorem}
\begin{proof}
  We recall that $k=|k|\be_1$ and $k^\perp=|k|\be_2$.
 The proof is performed iteratively on layers $t_{i-1}\leq x_1\leq t_i$ with $t_i=\frac iN$  for $0\leq i\leq N$ and $N=N(k)$ (to be determined) sufficiently large but finite for any given sufficiently smooth conductivity $\sigma(x)$. Here, $k=|k|\be_1$ is the vector in $\Rm^n$ used for the constructions of the CGO solutions. We define $y_i=(t_i,0,\ldots,0)$ for $0\leq i\leq N$.  Define two vectors close to $\be_1$ as 
   \begin{displaymath} 
  \bp=w\be_1+\sqrt{1-w^2}\be_2,\quad \bq=w\be_1-\sqrt{1-w^2}\be_2,
 \end{displaymath}
with $w<1$ sufficiently close to $1$ such that the light cones (for the Euclidean metric) emerging from $B_{n-1}(a)$ for the Lorentzian metric with main directions $\bp$ and $\bq$ still strictly include $\tilde X$; see Fig. \ref{fig:CGO}. 
All we need is that the radius $a$ be chosen sufficiently large so that any Lorentzian metric with direction close to $\be_1$, $\bp$, or $\bq$, has a light cone emerging from $B_{n-1}(a)$ that includes $\tilde X$. This means that any time-like trajectory (geodesic) from a point in $\tilde X$ crosses $B_{n-1}(a)$ for all metrics with direction close to $\be_1$, $\bp$ or $\bq$. See Fig. \ref{fig:CGO} where the light cone for $\bp$ is shown to strictly include $\tilde X$.
\begin{figure}[ht]
\begin{center}
\includegraphics[width=12cm]{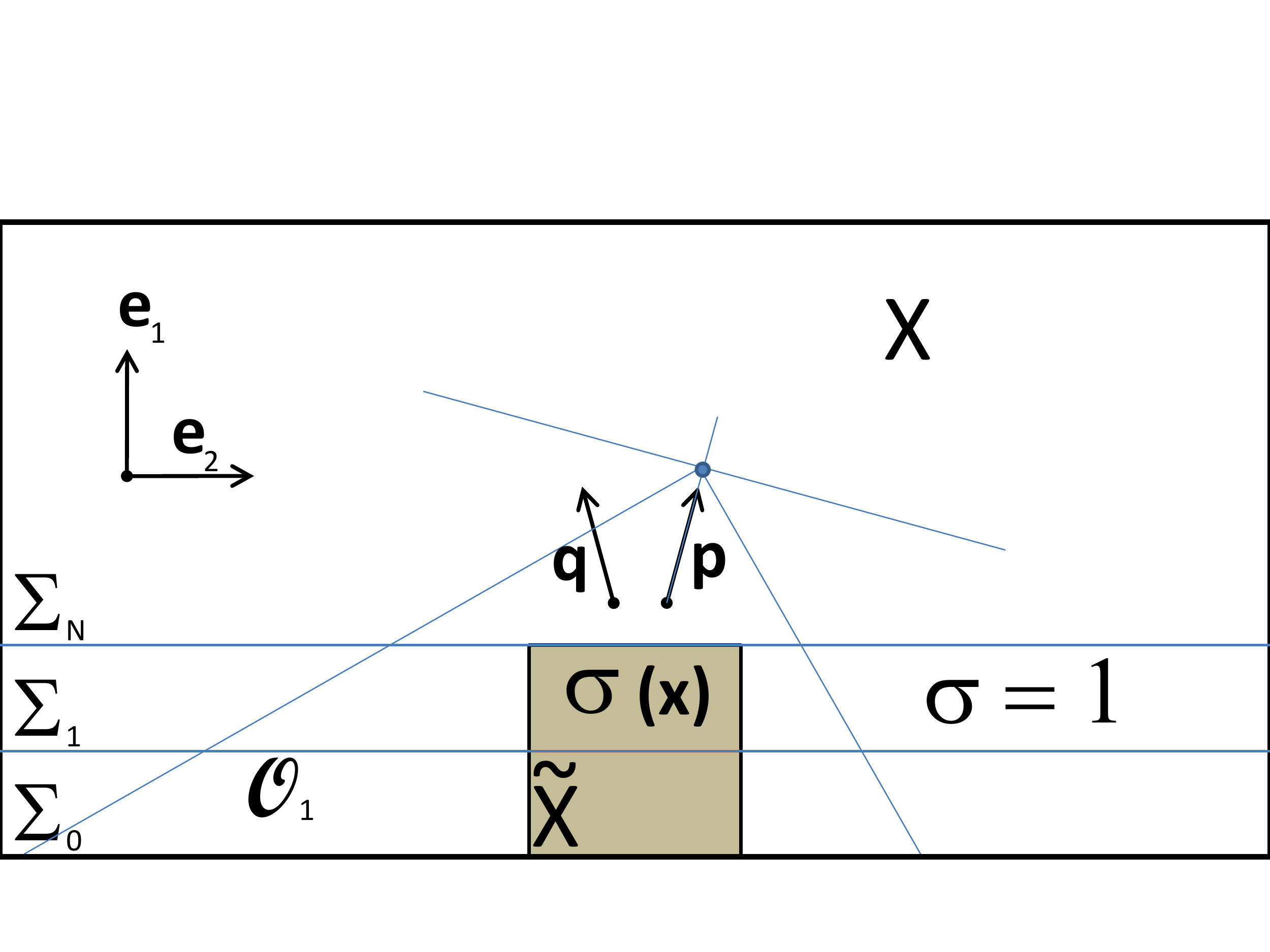}
\end{center}
\caption{Extended Geometry where complex geometric solutions are constructed.}
\label{fig:CGO}
\end{figure}
 
 Now consider the slab $t_0<x_1<t_1$. We prove a result on that slab and show that the Cauchy data at $t_1$ are controlled so that the same estimate may be used on $t_1<x_1<t_2$ and on all of $(0,1)$ by induction.  Let $\alpha_1$ and $\beta_1$ be the two angles in $(0,2\pi)$ such that 
 \begin{displaymath} 
  \cos\alpha_1 \theta(y_0) + \sin\alpha_1\theta^\perp(y_0)=\bp, \quad \cos\beta_1 \theta(y_0) + \sin\beta_1\theta^\perp(y_0)=\bq,
\end{displaymath}
where $\theta(x)\in\Sm^{n-1}$ is defined in \eqref{eq:spectheta}.

The complex geometric optics solutions are constructed as follows. We define $v=\Im e^{\rho\cdot x}$ and $w=\Re e^{\rho\cdot x}$ harmonic functions. Then we find that 
 \begin{displaymath} 
   \nabla v = e^{k^\perp\cdot x} |k| \theta(x), \quad \nabla w = e^{k^\perp\cdot x} |k| \theta^\perp(x),
 \end{displaymath}
so that for the two harmonic functions $v_1=\cos\alpha_1 v+\sin\alpha_1 w$ and $w_1=\cos\beta_1 v+\sin\beta_1 w$, we have on the slab $0<x_1<t_1$ that 
\begin{displaymath} \begin{array}{rclcl}
   \widehat{\nabla  v_1} &=&\cos\alpha_1 \theta(x)+\sin\alpha_1 \theta^\perp(x)  &=& \bp+O\Big(\dfrac{|k|}{N}\Big),\\[2mm]
   \widehat{\nabla  w_1} &=&\cos\beta_1 \theta(x)+\sin\beta_1 \theta^\perp(x)  &=& \bq+O\Big(\dfrac{|k|}{N}\Big).
   \end{array}
\end{displaymath}
For $t_1=\frac1N$ such that $|k|t_1=\frac{|k|}{N}$ is sufficiently small, $\widehat{\nabla  v_1}$ and $\widehat{\nabla  w_1}$ for all $x$ such that $0<x_1<t_1$ are two vector fields such that the associated Lorentzian metrics $\mh_{\widehat{\nabla  v_1}}$ and $\mh_{\widehat{\nabla  w_1}}$ have $\be_1$ as a time-like vector.

Let us now assume that $\sigma$ is arbitrary but smooth. The main idea of CGO solutions is that we can construct solutions for arbitrary $\sigma$ that are close to the solutions corresponding to $\sigma=1$ for $|k|$ sufficiently large. We construct CGO solutions $u_\rho$ of \eqref{eq:elliptic} (and $\tilde u_\rho$ by replacing $\sigma$ by $\tilde\sigma$) such that 
\begin{displaymath} 
    u_\rho = \dfrac{1}{\sqrt\sigma} e^{\rho\cdot x}(1+\psi_\rho),
 \end{displaymath}
with $|k|\psi_\rho$ bounded in the $C^1$ norm since $\sigma$ is sufficiently smooth by hypothesis. This result is proved in  \cite{BRUZ-IP-11} following earlier work in \cite{BU-IP-10}. These solutions are constructed on $\Rm^n$ and then restricted to $X$; their boundary condition $f_\rho$ is therefore specified by the construction. For such a solution, we find that
\begin{displaymath} 
  \nabla u_\rho = \dfrac{1}{\sqrt\sigma}e^{\rho\cdot x}|\rho| \big(\hat\rho + \phi_\rho\big),
 \end{displaymath}
where $|k||\phi_\rho|$ is also bounded in the uniform norm. 
 This shows that 
 \begin{displaymath} 
   \widehat {\nabla \Im u_\rho}(x) = \theta(x) + \phi_{\rho,i},\qquad \widehat {\nabla \Re u_\rho}(x) = \theta(x) + \phi_{\rho,r}
 \end{displaymath}
with $|k||\phi_{\rho,i}|$ and $|k||\phi_{\rho,r}|$ bounded in the uniform norm. As a consequence, we have constructed solutions of \eqref{eq:elliptic} with a gradient that is close to the prescribed $\theta(x)$ corresponding to harmonic functions.
 Construct now the two linear combinations
\begin{equation} \label{eq:vw1rho} 
   v_{1,\rho}=\cos\alpha_1 v_\rho+\sin\alpha_1 w_\rho,\,\,\, w_{1,\rho}=\cos\beta_1 v_\rho+\sin\beta_1 w_\rho,\quad v_\rho := \Im u_\rho,\,\,  w_\rho := \Re u_\rho.
\end{equation}
Knowledge of the Cauchy data for $v_{1,\rho}$ and $w_{1,\rho}$ is inherited from that for $v_\rho$ and $w_\rho$. Define $\tilde v_{1,\rho}$ and $\tilde w_{1,\rho}$ similarly with $\sigma$ replaced by $\tilde\sigma$. We choose $|k|$ sufficiently large and then $t_1|k|$ sufficiently small so that $\phi_\rho$ is a negligible vector that does not perturb the Lorentzian metric much and so that 
\begin{equation}\label{eq:errors} 
\widehat{\nabla v}_{1,\rho}=\bp+O(|k|t_1) +O(M|k|^{-1})\quad \mbox{ and } \quad \widehat {\nabla w}_{1,\rho}=\bq+O(|k|t_1)+O(M|k|^{-1})
 \end{equation}
 are directions of Lorentzian metrics for which (i) $\be_1$ is a time-like vector; and (ii) the light cone emerging from $B_{n-1}(a)$ includes $\tilde X$. Here, $M$ is the uniform bound of $\sigma$ in $H^{\frac n2+3+\eps}(\Rm^n)$ \cite{BRUZ-IP-11,BU-IP-10}. Note that this means that $t_1$ should be chosen on the order of $M|k|^{-2}$ once $|k|$ has been chosen so that $M|k|^{-1}$ is sufficiently small.

The same properties hold for the vectors constructed by replacing $\sigma$ by $\tilde\sigma$. Thus, the metric $\mg$ in \eqref{eq:mg} is given with $\alpha$ and $\beta$ close to $1$ and $\be(x)$ close to $\bp$ for the function $v_{1,\rho}$ and close to $\bq$ for the function $w_{1,\rho}$. Using Cauchy data on $\Sigma_0:=\{x_1=0\}$, we can then solve the linear equations on the slab $\cO_1:=\{0=t_0<x_1<t_1\}$ and get the solution at the surface $\Sigma_1:=\{x_1=t_1\}$. 
For the solutions $v_{1,\rho}$ and $w_{1,\rho}$, we obtain as a slight modification of \eqref{eq:localstab} \cite{Taylor-PDE-1} the stability result:
 \begin{equation}
  \label{eq:slabstab}
   \Big( \dint_{\Sigma_1} |f-\tilde f|^2 + |j-\tilde j|^2 \,d\sigma+ \dint_{\cO_1} E(dv) dx \leq C \Big( \dint_{\Sigma_0} |f-\tilde f|^2 + |j-\tilde j|^2 \,d\sigma + \dint_{\cO_1} |\nabla \delta H|^2 \,dx\Big).
\end{equation}
The above measurements $\delta H=H-\tilde H$ are those for the functions $(v_{1,\rho},\tilde v_{1,\rho})$ and  $(w_{1,\rho},\tilde w_{1,\rho})$. Such measurements can be constructed from the three measurements for $v_\rho$, $w_\rho$ and $v_\rho+w_\rho$. This is the place where we use the three measurements stated in the theorem: we need to ensure that $\sigma(x)|\mu \nabla v_\rho+\nu \nabla w_\rho|^2$ is available for any possible linear combination $(\mu,\nu)$ since the values of $\alpha_1$ and $\beta_1$ will vary (and will be called $\alpha_i$ and $\beta_i)$ on each slab $t_i<x_i<t_{i+1}$. Since the measurements $H$ for the $0-$Laplacian problem are {\em quadratic} in the elliptic solution, three measurements are sufficient by polarization to allow us to construct $\sigma(x)|\nabla v_{1,\rho}|^2$ and $\sigma(x)|\nabla w_{1,\rho}|^2$.

On $\Sigma_1$, we have control on the Cauchy data of $v_{1,\rho}$ and $w_{1,\rho}$ and hence of  $v_\rho=\Im u_\rho$ and $w_\rho=\Re u_\rho$ thanks to \eqref{eq:slabstab} and \eqref{eq:vw1rho}. Here, we need that $\bp$ and $\bq$ be not too close to one-another (this is guaranteed by $w<1$) so that the inversion of the $2\times2$ system is well-conditioned. On each slab, we define the angles $\alpha_i$ and $\beta_i$ in order again to have Lorentzian metrics with directions close to $\bp$ and $\bq$. We then obtain a similar estimate to \eqref{eq:slabstab} and pursue by induction until we reach the slab $\cO_N:=\{t_{N-1}<x_1<t_N=1\}$. 

The stability results then apply to $\Im u_\rho$ and $\Re u_\rho$, and we thus obtain a global estimate for $\sigma$ as in earlier sections. So far, the illuminations $f$ prescribed on $X$ to solve the elliptic problem are of a very specific type. In order for $\Im u_\rho$ and $\Re u_\rho$ to be the solutions to the elliptic problems on $X$, then $(f_1,f_2)$ needs to be the trace of $(\Im u_\rho,\Re u_\rho)$ on $\partial X$. It is for these illuminations that the three measurements $H_{ij}(x)$ for $(i,j)\in I$ generate Lorentzian metrics that satisfy the above sufficient properties. Since $\sigma$ is not known, these traces are not known either.

However, any Lorentzian metric that is sufficiently close to the Lorentzian metrics constructed with the real and imaginary parts of $u_\rho$ will inherit the same light cone properties and, in particular, the fact that $\be_1$ is a time-like vector for these new Lorentzian metrics throughout $X=(0,1)\times B_{n-1}(a)$. Therefore, there is an open set of boundary conditions $(f_1,f_2)$ close to $(\Im u_\rho|_{\partial X},\Re u_\rho|_{\partial X})$  such that the conclusion \eqref{eq:slabstab} holds, as well as the same expressions on the other slabs $\cO_i$.
This concludes the proof of the result.
\end{proof}
\begin{remark}\rm
  \label{rem:cgo0}  The ``three'' measurements $H_{ij}$ for $(i,j)\in I$ in \eqref{eq:threemeas} actually correspond to ``two'' physical measurements. Indeed, we can replace $u_{\eps}$ by $u_{1;\eps}$ and $u_{-\eps}$ by $u_{2;-\eps}$  in \eqref{eq:Green} and obtain in the limit $\sigma \nabla u_1\cdot\nabla u_2$, which combined with $H_{11}$ and $H_{22}$ yields $H_{12}$ defined in \eqref{eq:threemeas}. The experimental acquisition of $H_{11}$ is in fact sufficient to also acquire $\sigma \nabla u_1\cdot\nabla u_2$ as demonstrated in \cite{KK-AET-11}. 
\end{remark}
\begin{remark}\rm
  \label{rem:cgo1} The above theorem shows a uniqueness and stability result for arbitrary, sufficiently smooth, conductivities. However, the boundary conditions $f$ are quite specific since they need to be sufficiently close to non-explicit, $\sigma$-dependent, traces of complex geometrical optics solutions.   In some sense, the difficulty inherent to the spatially varying Lorentzian metric $\mh(x)$ in \eqref{eq:mh} has been shifted to the difficulty of constructing adapted boundary conditions (illuminations).
  
Note that the condition of flatness of the surfaces $\Sigma_i$ in the above construction is not essential. Surfaces with a geometry such as that depicted in Fig. \ref{fig:localconst} may also be considered. Such surfaces allow us to reduce the size of the domain $X$ on which the conductivity $\sigma=1$ needs to be extended. Unless the domain $X$ has a specific geometry similar to that of the domain $\cO$ between $\Sigma_1$ and $\Sigma_2$ in Fig.  \ref{fig:localconst}, it seems necessary to augment the size of $\tilde X$ to that of $X$ as described above to obtain a global uniqueness result.
\end{remark}

%
\section*{Acknowledgment} Partial support from the National Science Foundation is greatly acknowledged.
 
{\small

}

\end{document}